\documentclass{amsart}

\usepackage[letterpaper,margin=1in]{geometry}
\usepackage[foot]{amsaddr}

\usepackage{amsmath,amssymb,amsthm}
\usepackage{hyperref}
\usepackage{mathtools}
\usepackage{enumerate}
\usepackage{bbm}
\usepackage{tikz}
\usetikzlibrary{calc}

\newtheorem{introthm}{Theorem}

\newtheorem{theorem}{Theorem}[section]
\newtheorem{proposition}[theorem]{Proposition}
\newtheorem{lemma}[theorem]{Lemma}

\theoremstyle{definition}

\newtheorem{notation}[theorem]{Notation}

\theoremstyle{remark}
\newtheorem{remark}[theorem]{Remark}
\newtheorem{example}[theorem]{Example}

\numberwithin{equation}{section}

\newcommand{\R}{\mathbb{R}}
\newcommand\C{\mathbb{C}}

\DeclareMathOperator{\Tr}{Tr}

\DeclareMathOperator{\id}{id}

\DeclareMathOperator{\Ran}{Ran}
\DeclareMathOperator{\Span}{Span}
\DeclareMathOperator{\diag}{diag}

\DeclarePairedDelimiter{\ip}{\langle}{\rangle}
\DeclarePairedDelimiter{\ang}{\langle}{\rangle}
\DeclarePairedDelimiter{\norm}{\lVert}{\rVert}

\DeclarePairedDelimiter{\paren}{(}{)}
\DeclarePairedDelimiter{\abs}{|}{|}
\DeclarePairedDelimiter{\set}{\{}{\}}

			\newcommand{\gp}{\mathop{\begin{tikzpicture}[baseline]
						\node[circle, fill=cyan, draw=black, inner sep=0pt, minimum size=3pt] (mid) at (0, 2.5pt) {};
						\foreach \x in {0,60,...,300} {
							\node[circle, fill=cyan, draw, inner sep=0pt, minimum size=3pt] (\x) at ($(mid)!6pt!(mid.\x)$) {};
							\draw (\x) -- (mid) ;
						}
			\end{tikzpicture}}}
			
			\def\semicolon{;}
			\def\applytolist#1{
				\expandafter\def\csname multi#1\endcsname##1{
					\def\multiack{##1}\ifx\multiack\semicolon
					\def\next{\relax}
					\else
					\csname #1\endcsname{##1}
					\def\next{\csname multi#1\endcsname}
					\fi
					\next}
				\csname multi#1\endcsname}

			\def\calc#1{\expandafter\def\csname c#1\endcsname{{\mathcal #1}}}
			\applytolist{calc}QWERTYUIOPLKJHGFDSAZXCVBNM;
			\def\bbc#1{\expandafter\def\csname b#1\endcsname{{\mathbb #1}}}
			\applytolist{bbc}QWERTYUIOPLKJHGFDSAZXCVBNM;
			\def\fc#1{\expandafter\def\csname f#1\endcsname{{\mathfrak #1}}}
			\applytolist{fc}QWERTYUIOPLKJHGFDSAZXCVBNM;
			\def\rc#1{\expandafter\def\csname r#1\endcsname{{\mathrm #1}}}
			\applytolist{rc}QWERTYUIOPLKJHGFDSAZXCVBNM;
			\def\ttc#1{\expandafter\def\csname t#1\endcsname{{\mathtt #1}}}
			\applytolist{ttc}QWERTYUIOPLKJHGFDSAZXCVBNM;
			
			\makeatletter
			\def\moverlay{\mathpalette\mov@rlay}
			\def\mov@rlay#1#2{\leavevmode\vtop{%
					\baselineskip\z@skip \lineskiplimit-\maxdimen
					\ialign{\hfil$#1##$\hfil\cr#2\crcr}}}
			\makeatother

			\title[Atoms of graph products]{The atoms of graph product von Neumann algebras}
			
\begin{document}
				
				\author{Ian Charlesworth$^\circ$}
				\address{$^\circ$School of Mathematics, Cardiff University, United Kingdom}
				\email{$^\circ$\href{mailto:charlesworthi@cardiff.ac.uk}{charlesworthi@cardiff.ac.uk}}
				
				\author{David Jekel$^\bullet$}
				\address{$^\bullet$Department of Mathematical Sciences, University of Copenhagen, Denmark}
				\email{$^\bullet$\href{mailto:daj@math.ku.dk}{daj@math.ku.dk}}
				
				
				\begin{abstract}
					We completely classify the atomic summands in a graph product $(M,\varphi) = \gp_{v \in \cG} (M_v,\varphi_v)$ of von Neumann algebras with faithful normal states.  Each type I factor summand $(N,\psi)$ is a tensor product of type I factor summands $(N_v,\psi_v)$ in the individual algebras.  The existence of such a summand and its weight in the direct sum can be determined from the $(N_v,\psi_v)$'s using explicit polynomials associated to the graph.
				\end{abstract}
				
				\maketitle

				
					
				
				\section{Introduction}
				
				\emph{Graph products of groups} were introduced by Green \cite{Green1990}: given a graph $\cG$ and groups $\Gamma_v$ for each vertex $v$, the graph product $\Gamma = \gp_{v \in \cG} \Gamma_v$ is the free product of the $\Gamma_v$'s modulo the relations that $\Gamma_v$ and $\Gamma_{v'}$ commute when $v$ is adjacent to $v'$.  Analogous constructions for operator algebras have been introduced several times, at first from a probabilistic viewpoint as providing a mixture of classical and free independence \cite{Mlot2004,SpWy2016}, later as providing the natural analogue of Green's construction in the setting of $\mathrm{C}^*$ and von Neumann algebras \cite{CaFi2017}, and again as an asymptotic description of certain random many body systems \cite{MorLau2019}. 
				As we will see below, the study of the combinatorics of words needed for graph products actually preceded the introduction of graph products of groups \cite{CF1969}.

				In this paper, we focus on the setting of finite graphs, and our goal is to completely describe the finite-dimensional direct summands in graph product von Neumann algebras, motivated by various results in the case of free products (which correspond to graphs with no edges).  Dykema \cite{Dykema1993freedimension} gave a precise description of free products of amenable tracial von Neumann algebras as a direct sum of a $\mathrm{II}_1$ factor with (sometimes) various finite-dimensional pieces.  After many partial results were proved, Ueda \cite{UedaTypeIIIfreeproduct} finally gave a complete classification of when a free product of general von Neumann algebras is a factor, when it is diffuse, and what the types of its summands are.    The present authors and their collaborators also classified when a graph product is a factor under an additional assumption that each algebra has a state-zero unitary in the centralizer of the state \cite{tesseract}, but this is precisely a situation when a free product will always be diffuse.
				In the other extreme, Raum and Skalski \cite{RaumSkalski2023} classified the atoms of the von Neumann algebras of multiparameter Hecke algebras, which are graph products of two-dimensional algebras (see \cite[Corollary 3.4]{Caspers2020}), generalizing earlier works on factoriality \cite{Garncarek2016,CKL2021}.
				
				An illustrative subproblem is to determine when the intersection of graph-independent nonzero projections is nonzero (this is closely related to the existence of Hecke eigenvectors in the Hecke algebra setting of \cite{RaumSkalski2023}).  
				
				\begin{introthm} \label{thm: intersection of projections}
					Let $\cG = (\cV ,\cE )$ be a finite graph, and let $(M,\varphi) = \gp_{v \in \cG} (M_v,\varphi_v)$ be a graph product of von Neumann algebras with faithful normal states (see \S \ref{sec: preliminaries} for definition).  For each $v$, let $p_v$ be a projection in $M_v$.  For a graph $\cG$, define
					\begin{equation} \label{eq: main polynomials}
						\fK_{\cG}((x_v)_{v \in \cV}) = \sum_{\substack{\cK \subseteq \cG\\\text{clique}}} (-1)^{|\cK|} \prod_{v \in \cK} x_v,
					\end{equation}
					where the sum ranges over all cliques $\cK$ in the graph $\cG$; by convention $\cK = \varnothing$ is a clique and the corresponding term in the sum is $1$ (the empty product).  Then
					\[
					\bigwedge_{v \in \cV} p_v \neq 0 \iff \forall \cV' \subseteq \cV, \quad \fK_{\cG'}((1-\varphi_v(p_v))_{v \in \cV'}) > 0, 
					\]
					where $\cG'$ is the subgraph of $\cG$ induced by $\cV'$.  Moreover, in this case,
					\[
					\varphi \left( \bigwedge_{v \in \cV} p_v \right) = \fK_{\cG}((1-\varphi_v(p_v))_{v \in \cV}).
					\]
				\end{introthm}
				
				First, we remark that the theorem recovers the known characterizations in the free product and tensor product settings.
				
				\begin{example}
					Suppose that $\cG$ is a complete graph (i.e., all vertices are adjacent).  Then the graph product reduces to a tensor product $\bigotimes_{v \in \cV} (M_v,\varphi_v)$ and so $\bigwedge_{v \in \cV} p_v$ will always be nonzero provided each $p_v$ is.
					For each $\cV' \subseteq \cV$, all subsets of $\cV'$ are cliques, and therefore using binomial expansion,
					\[
					\fK_{\cG'}((x_v)_{v \in \cV'}) = \sum_{\cK \subseteq \cV'} (-1)^{|\cK|} \prod_{v \in \cK} x_v = \prod_{v \in \cV'} (1 - x_v).
					\]
					Thus, $\fK_{\cG'}((x_v)_{v \in \cV'})$ is automatically positive when $x_v = 1 - \varphi(p_v) \in [0,1)$.
				\end{example}
				
				\begin{example}
					Suppose that $\cG$ has no edges.  Then the graph product is simply the free product $*_{v \in \cV} (M_v,\varphi_v)$.  Since the only non-empty cliques in $\cG$ are singletons, we have for $\cV' \subseteq \cV$
					\[
					\fK_{\cG'}((x_v)_{v \in \cV'}) = 1 - \sum_{v \in \cV'} x_v.
					\]
					Therefore, in this case, whenever $x_v \in [0,1)$, positivity of $\fK_{\cG}((x_v)_{v \in \cV})$ automatically implies positivity of $\fK_{\cG'}((x_v)_{v \in \cV})$ for every $\cV' \subseteq \cV$.  Hence, we see that $\bigwedge_{v \in \cV} p_v \neq 0$ precisely when $\sum_{v \in \cV} (1 - \varphi(p_v)) < 1$, which recovers the well-known characterization which follows from inductive application of \cite[Theorem 1.1]{Dykema1993freedimension}.
				\end{example}
				
				At this point, it may not be clear why it is necessary to look at the clique polynomials $\fK_{\cG'}$ for induced subgraphs $\cG'$, and not simply the single polynomial $\fK_{\cG}$.  We first point out that if $\bigwedge_{v \in \cV} p_v \neq 0$, then $\bigwedge_{v \in \cV'} p_v \neq 0$ for every $\cV' \subseteq \cV$.
				Thus, whatever characterization we provide for non-vanishing of the projection, if it holds for a graph $\cG$ it must also hold for every induced subgraph.
				The following example illustrates that positivity of $\fK_{\cG}(x)$ alone does not always imply positivity of $\fK_{\cG'}(x|_{\cV'})$ for induced subgraphs.
				
				\begin{example}
					Let $\cG = (\cV,\cE)$ be a graph with $n$ vertices which is $3$-regular (i.e., each vertex has exactly $3$ neighbors) and triangle-free (i.e., there do not exist vertices $i$, $j$, $k$ with $i \sim j \sim k \sim i$).  For instance, $\cG$ could be the complete bipartite graph $K_{3,3}$, the Petersen graph, or the graph given by the vertices and edges of a dodecahedron.  Let $\alpha \in [2/3,1)$, and suppose that $x_v = \alpha$ for all $v \in \cV$, and write $x = (x_v)_{v \in \cV}$. Since $\cG$ is triangle-free, there are no cliques of size greater than two, so
					\[
					\fK_{\cG}(x) = 1 - |\cV| \alpha + |\cE| \alpha^2 = 1 - n \alpha + \frac{3n}{2} \alpha^2 \geq 1.
					\]
					On the other hand, if we consider two vertices $v$ and $w$ that are not adjacent and take $\cV' = (v,w)$, the
					\[
					\fK_{\cG'}(x_v,x_w) = 1 - 2\alpha \leq 1 - 2(2/3) < 0.
					\]
					Thus, $\fK_{\cG}(x) > 0$ and $\fK_{\cG'}(x_v,x_w) < 0$.  In particular, if we are given projections in a graph product as in Theorem \ref{thm: intersection of projections} with $\varphi(p_v) = 1 - \alpha$, then $p_v \wedge p_w = 0$ and so $\bigwedge_{u \in \cV} p_u = 0$ also.
				\end{example}
				
				
				\begin{remark}
					Although $\fK_{\cG}((x_v)_{v \in \cV}) > 0$ is not sufficient to imply that $\fK_{\cG'}((x_v)_{v \in \cV'}) > 0$ for $\cV' \subseteq \cV$, we can state several sufficient conditions in terms of $\fK_{\cG'}$ alone.  Namely, in \S \ref{subsec: region}, we show that it is sufficient to check that $\fK_{\cG}((\alpha x_v)_{v \in \cV}) > 0$ for all $\alpha \in [0,1]$.  Moreover, letting
					\[
					R(\cG) = \{x \in [0,1]^{\cV}: \fK_{\cG'}((x_v)_{v \in \cV'}) > 0 \text{ for all } \cV' \subseteq \cV \},
					\]
					we will show that $R(\cG)$ is the connected component of the region where $\fK_{\cG} > 0$ containing $0$.
				\end{remark}
				
				
				Now let us state our full result on finite-dimensional factor summands.  In the following $\mathbb{M}_n$ denotes the $n \times n$ matrix algebra $M_n(\mathbb{C})$.  The classification uses the same clique polynomials as in Theorem \ref{thm: intersection of projections}, but now applied to parameters computed from the state on each of the matrix algebra summands.
				
				\begin{introthm} \label{thm: classification of finite-dimensional summands}
					Let $\cG$ be a finite graph, let $(M_v,\varphi_v)$ be von Neumann algebras with faithful normal states, and let $(M,\varphi)$ be the graph product.  For each $v$, suppose that $(M_v,\varphi_v)$ has a direct summand $(\mathbb{M}_{n(v)},\psi_v)$ with weight $\alpha^{(v)}$, and suppose the state $\psi_v$ is given by
					\[
					\psi_v(A) = \Tr_{n(v)}(A \diag(\lambda_1^{(v)},\dots,\lambda_{n(v)}^{(v)})), \text{ for } A \in \mathbb{M}_{n(v)}.
					\]
					Let
					\[
					s^{(v)} = \sum_{j=1}^{n(v)} \frac{1}{\alpha^{(v)} \lambda_j^{(v)}}.
					\]
					Suppose that $\cK = \{v \in \cV: n(v) > 1\}$ is a clique in $\cG$ and that $\fK_{\cG'}((1 - 1/s^{(v)})_{v \in \cV'}) > 0$ for all $\cV' \subseteq \cV$, where $\fK_{\cG'}$ is as in \eqref{eq: main polynomials}.  Then $(M,\varphi)$ has a direct summand
					\[
					(N,\psi) = \bigotimes_{v \in \cK} (\bM_{n(v)},\psi_v)
					\]
					with weight
					\[
					\alpha = \prod_{v \in \cV} \alpha^{(v)} s^{(v)} \; \fK_{\cG}((1 - 1/s^{(v)})_{v \in \cV}) = \prod_{v \in \cV} \alpha^{(v)} \sum_{\substack{\cK \subseteq \cV  \\ \text{clique}}} \prod_{v \in \cV  \setminus \cK} s^{(v)} \prod_{v \in \cK} (1 - s^{(v)}).
					\]
					(For a precise description of $(N,\psi)$, see \S \ref{subsec: construction of matrix units}.)  Moreover, all finite-dimensional factor summands of $(M,\varphi)$ arise in this way.
				\end{introthm}
				
				\begin{remark}
					If $\fK_{\cG'}((1 - 1/s^{(v)})_{v \in \cV'}) > 0$ for all $\cV' \subseteq \cV$, then the condition that $\cK = \{v \in \cV: n(v) > 1\}$ is a clique is automatic.  Indeed, if $n(v) > 1$, then $s^{(v)} \geq 2$ since $\alpha^{(v)}$ and $\lambda_j^{(v)}$ are in $(0,1)$.  Thus, we have $1 - 1/s^{(v)} \geq 1/2$ for $v \in \cK$.  Now if $v$ and $v'$ were in $\cK$ and not adjacent, then taking $\cV' = \{v,v'\}$, we would have
					\[
					\fK_{\cG'}((1 - 1/s^{(w)})_{w \in \cV'}) = 1 - (1 - 1/s^{(v)}) - (1 - 1/s^{(v)}) \leq 1 - 1/2 - 1/2 = 0,
					\]
					which contradicts our assumption of positivity of $\fK_{\cG'}$.
				\end{remark}
				
				\begin{example}
					Let us give a sample computation with the smallest graph for which the graph product cannot be expressed in terms of tensor and free products.

					Consider the path graph $\cP_4$ with vertices $1,2,3,4$ and edges $\{1,2\}$, $\{2,3\}$, and $\{3,4\}$, and the following algebras:
					\begin{itemize}
						\item $(M_1,\varphi_1) = (M_4,\varphi_4) = \C \oplus \C$ with weight $23/25$ on the first summand and $2/25$ on the second summand.  We take $N_1 = N_4$ to be the first summand.  Thus,
						\[
						\alpha^{(1)} = \alpha^{(4)} = 23/25, \quad 
						\lambda_1^{(1)} = \lambda_1^{(4)} = 1, \quad s^{(1)} = s^{(4)} = 25/23.
						\]
						\item $(M_2,\varphi_2) = (M_3,\varphi_3) = \mathbb{M}_2$ with the state given by $A \mapsto \Tr(A \diag(3/5,2/5))$.  We take $N_2 = M_2$ and $N_3 = M_3$.  Thus,
						\[
						\alpha^{(2)} = \alpha^{(3)} = 1, \quad \lambda_1^{(2)} = \lambda_1^{(3)} = 3/5, \quad \lambda_2^{(2)} = \lambda_2^{(3)} = 2/5, \quad s^{(2)} = s^{(3)} = 25/6.
						\]
					\end{itemize}
					Now let $x_v = 1 - 1 / s^{(v)}$, namely
					\[
					x_1 = x_4 = \tfrac{2}{25}, \qquad x_2 = x_3 = \tfrac{19}{25}.
					\]
					Then
					\begin{align*}
						\fK_{\cP_4}(x) 
						&= 1 - x_1 - x_2 - x_3 - x_4 + x_1x_2 + x_2x_3 + x_3x_4 \\
						&=  1 - \frac{2}{25} - \frac{19}{25} - \frac{19}{25} - \frac{2}{25} + \frac{2}{25} \cdot \frac{19}{25} + \frac{19}{25} \cdot \frac{19}{25}  + \frac{19}{25} \cdot \frac{2}{25}
						= \frac{12}{625} > 0.
					\end{align*}
					One can check that $\fK_{\cG'}(x) > 0$ for every induced subgraph, too.  (In fact it turns out that $\fK_{\cP_4}(y) > 0$ implies $\fK_{\cG'}(y) > 0$ for every induced subgraph $\cG'$ of $\cP_4$ and every $y \in (0,1)^4$, although showing this takes a bit of work, with the hardest case corresponding to the subgraph induced by $\{1,4\}$.)
					Theorem \ref{thm: classification of finite-dimensional summands} now yields a summand of the form $(N_2,\psi_2) \otimes (N_3,\psi_3)$, or $\mathbb{M}_4$ with state given by $\diag(9/25,6/25,6/25,4/25)$.  Its weight in the direct sum decomposition is
					\[
					\alpha = \prod_{v \in \cV} \alpha^{(v)} s^{(v)} \fK_{\cP_4}(x) = 1 \cdot \frac{25}{6} \cdot \frac{25}{6} \cdot 1 \cdot \frac{12}{625}
					= \frac{1}{3}.
					\]
				\end{example}
				
				Our result in Theorem \ref{thm: classification of finite-dimensional summands} extends to infinite-dimensional $B(H)$ summands as follows.  Here we add the assumption that $M_v \neq \C$ to avoid annoying trivial cases later on.  Of course, in general, we can consider the graph $\cG'$ obtained by deleting all vertices where $M_v = \C$.  Then $\gp_{v \in \cG} (M_v,\varphi_v)$ is canonically isomorphic to $\gp_{v \in \cG'} (M_v,\varphi_v)$, for which the theorem below applies.
				
				\begin{introthm} \label{thm: classification of B(H) summands}
					Let $\cG$ be a finite graph, let $(M_v,\varphi_v)$ be von Neumann algebras with faithful normal states with $M_v \neq \C$, and let $(M,\varphi)$ be the graph product.
					Suppose that $(N,\psi)$ is an infinite-dimensional type I factor direct summand of $(M,\varphi)$.
					Then there is a graph join decomposition $\cG = \cG_1 + \cG_2$ (see \S \ref{subsec: graphs}) where $\cG_1$ is a (non-empty) complete graph but $\cG_2$ is allowed to be empty.
					For every vertex $v \in \cV_1$, there is an infinite-dimensional type I factor summand $(N_v,\psi_v)$ in $(M_v,\varphi_v)$.
					For every vertex $v \in \cG_2$, there is a finite-dimensional type I factor summand $(N_v,\psi_v)$.  The given summand $(N,\psi)$ arises as the tensor product of $\bigotimes_{v \in \cV_1} (N_v,\psi_v)$ with a finite-dimensional factor direct summand in $\gp_{v \in \cG_2} (M_v,\varphi_v)$ described by Theorem \ref{thm: classification of finite-dimensional summands}.
				\end{introthm}
				
				Let us sketch the proof technique, focusing on the subproblem in Theorem \ref{thm: intersection of projections}.  In the setting of Theorem \ref{thm: intersection of projections}, the projection $\bigwedge_{v \in \cV} p_v$ must be in the von Neumann algebra generated by $(p_v)_{v \in \cV}$ which is the graph product of $\C p_v \oplus \C (1-p_v)$ for $v \in \cV$.  Hence, we assume without loss of generality that $M_v$ is generated by $p_v$ (thus $\C p_v$ is a one-dimensional summand of $M_v$ so we are in the situation of Theorem \ref{thm: classification of finite-dimensional summands}).  We describe $\bigwedge_{v \in \cV} p_v$ by analyzing its range in the Hilbert space $L^2(M,\varphi)$ directly, after decomposing $L^2(M,\varphi)$ in terms of $\cG$-reduced words (see \S \ref{subsec: words} for definitions).
				
				
				We fix a set of equivalence class representatives $\mathcal{W}$ for the $\cG$-reduced words and write $\mathcal{W}_\ell$ for the elements of length $\ell$.  The construction of graph products implies that $L^2(M,\varphi)$ can be written as an orthogonal direct sum
				\[
				L^2(M,\varphi) = \C 1 \oplus \bigoplus_{\ell \in\bN} \bigoplus_{w \in \mathcal{W}_\ell} \overline{\Span}\left\{ x_1 \dots x_\ell: x_j \in M_{w_j}, \varphi_{w_j}(x_j) = 0 \right\}.
				\]
				Since we took $M_v = \C p_v \oplus \C(1 - p_v)$, in this case each summand is one-dimensional since $\ker(\varphi_v) = \Span(\mathring{p}_v)$ where $\mathring{p}_v = p_v - \varphi_v(p_v)$.  If $\xi \in \bigcap_{v \in \cV} \Ran(p_v)$, then one can solve for the terms in the direct sum composition by induction on the length of $w$, which results in
				\[
				\xi = \alpha \sum_{\ell \in\bN_0} \sum_{w \in \mathcal{W}_\ell} \frac{1}{\varphi_{w_1}(p_{w_1}) \dots \varphi_{w_{\ell}}(p_{w_{\ell}})} \mathring{p}_{w_1} \dots \mathring{p}_{w_{\ell}}
				\]
				for some constant $\alpha$.  This computation is parallel to the computation of Hecke eigenvectors for Hecke algebras in \cite[\S 3]{RaumSkalski2023}, though the general computation for Theorem \ref{thm: classification of finite-dimensional summands} is more involved.  Since $\norm{\mathring{p}_v}_{L^2(M_v,\varphi_v)}^2 = \varphi_v(p_v)(1 - \varphi_v(p_v))$, we discover that
				\begin{equation} \label{eq: power series for norm}
					\norm{\xi}^2 = \abs{\alpha}^2 \sum_{\ell \in\bN_0} \sum_{w \in \mathcal{W}_\ell} \prod_{i=1}^\ell \frac{1 - \varphi_{w_i}(p_{w_i})}{\varphi_{w_i}(p_{w_i})}.
				\end{equation}
				In particular, there is a nonzero $\xi$ in the range of $p = \bigwedge_{v \in \cV} p_v$ if and only if this series converges.  In that case, $p$ itself, viewed as an element of $L^2(M,\varphi)$, is such a vector $\xi$ with $\alpha = \varphi(p)$ and $\norm{\xi}_{L^2(M,\varphi)}^2 = \varphi(p)$, which allows us to solve for $\varphi(p)$ as the reciprocal of the infinite series.
				
				Hence, it is crucial to understand the power series
				\begin{equation} \label{eq: power series f}
					f((X_v)_{v \in \cV}) = \sum_{\ell \in\bN_0} \sum_{w \in \mathcal{W}_\ell} X_{w_1} \dots X_{w_{\ell}};
				\end{equation}
				this is the same as the weighted growth series for a Coxeter system used in the context of Hecke algebras; see \cite[Theorem A]{RaumSkalski2023}.
				A closely related power series is
				\begin{equation} \label{eq: power series tilde f}
					\tilde{f}((X_v)_{v \in \cV}) = f((X_v/(1-X_v))_{v \in \cV}), 
				\end{equation}
				which also gives a sum over words up to equivalence, but this time allowing the same letter to repeat arbitrarily, since $X_v / (1 - X_v) = \sum_{k \in \bN} X_v^k$.  So $\tilde{f}((X_v)_{v \in \cV})$ gives the sum over not-necessarily-reduced words up to equivalence.  Fortunately, this power series has come up before.
				
				In \cite{CF1969}, Cartier and Foata introduced and studied partially commutative monoids generated by a finite set of generators subject only to some commutation relations between the generators.
				They used these to study flows on directed graphs
				and to give a version of MacMahon's Master Theorem with coefficients in a non-commutative ring.
				These partially commutative monoids are sometimes referred to as trace monoids or Cartier--Foata monoids, and have applications across mathematics and computer science: see, for example, \cite{MR3071084, MR1478993, MR1110852}.
				The main identity established in \cite{CF1969}, which is of direct use to us, is an expression of the power series $\tilde f( (X_v)_{v \in \cV})$ as the inverse of a polynomial:
				\[
				\tilde{f}((X_v)_{v \in \cV}) = \left( \sum_{\substack{\cK \subseteq \cG\\\text{clique}}} (-1)^{|\cK|} \prod_{v \in \cK} X_v \right)^{-1}.
				\]
				This is treated as an equation of formal power series with partially commuting indeterminates, but it remains true with commuting indeterminates.
				We remark that the formula was recently used in \cite{OSTY2025} to study the operator norms in graph products.
				We also remark that the power series $\tilde f( (X_v)_{v \in \cV})$ represents a rational function, and this rationality is closely related to the good algorithmic behavior of graph products of groups \cite{HM1995}.
				
				We convert the formula for $\tilde{f}$ above into a formula for $f$ by
				\[
				f((X_v)_{v \in \cV}) = \tilde{f}((X_v/(1+X_v))_{v \in \cV}) = \left( \sum_{\substack{\cK \subseteq \cG\\\text{clique}}} (-1)^{|\cK|} \prod_{v \in \cK} \frac{X_v}{1+X_v} \right)^{-1},
				\]
				and obtain
				\[
				f\left(\left(\frac{1 - \varphi_v(p_v)}{\varphi_v(p_v)}\right)_{v \in \cV}\right) = f\left(\left(\frac{1}{\varphi_v(p_v)} - 1 \right)_{v \in \cV}\right) = \left( \sum_{\substack{\cK \subseteq \cG\\\text{clique}}} (-1)^{|\cK|} \prod_{v \in \cK} (1 - \varphi_v(p_v)) \right)^{-1},
				\]
				which is exactly $\fK_{\cG}((1 - \varphi_v(p_v))_{v \in \cV})^{-1}$. 
				To determine when the series \eqref{eq: power series for norm} converges, we analyze the convergence of \eqref{eq: power series tilde f} on real positive inputs (see \S \ref{subsec: power series formula}, especially Lemma \ref{lem: convergence criteria}).
				The series has positive coefficients, and it agrees with the reciprocal of a polynomial on the region of convergence containing zero; thus, if we increase the values of the inputs from zero, the series will remain convergent until the polynomial vanishes, so the convergence conditions will naturally be stated in terms of positivity of the clique polynomials.
				We also compare the series with other quantities obtained by zeroing out some of the variables, which produces the series corresponding to subgraphs of $\mathcal{G}$ and eventually results in the conditions stated in our theorems.
				
				
				The proof of Theorem \ref{thm: classification of finite-dimensional summands} uses a generalization of the Hilbert space computation sketched above. Here we need to study not just a single projection but matrix units obtained from appropriate intersections of the diagonal matrix units in $\mathbb{M}_{n(v)}$ and their conjugates by partial isometries (see \S \ref{subsec: construction of matrix units}).  The construction of these matrix units is motivated by Dykema's work in \cite[Proposition 3.2]{Dykema1993freedimension} and Ueda's in \cite[Theorem 4.1]{UedaTypeIIIfreeproduct} for free products.  One then needs to show that these are the only finite-dimensional factor summands in $(M,\varphi)$ (see \S \ref{subsec: realization of all}).  We remark that the free product case considered by Dykema and Ueda is much simpler because, first of all, one can reduce to the case of a free product of two von Neumann algebras and, furthermore, in each finite-dimensional summand there can only be one $v$ for which $n(v) > 1$.  Moreover, our approach to showing that these are the only finite-dimensional summands in \S \ref{subsec: realization of all} is completely different and avoids any analysis of the non-atomic part and any induction on the number of summands in $M_v$.
				
				Since we have classified the type I factor summands in graph products, it is natural to seek a full type decomposition for the diffuse part analogous to Ueda's theorem.  Of course, the amount of casework needed for graph products is even more than for free products.  For instance, as seen in \cite{tesseract}, graph products multiply the occurrences of the annoying exceptional case that the free product of two $2$-dimensional algebras can have an $\mathbb{M}_2 \otimes L^\infty[0,1]$ summand (and this case arises naturally from right-angled Coxeter groups).  We leave the full type classification as a problem for future work.
				
				We finally remark on the related problem of finding the atoms of the spectral distribution $\sum_{v \in \cV} x_v$ where $x_v$ is a self-adjoint operator in $(M_v,\varphi_v)$.  If $\mu_v$ is the spectral measure of $x_v$ with respect to the state $\varphi_v$, then the spectral measure $\mu$ of the operator $x = \sum_{v \in \cV} x_v$ can be viewed as a ``graph convolution'' of the measure $(\mu_v)_{v \in \cV}$.  Indeed, in the case of a complete graph, $\mu$ is the classical convolution of $(\mu_v)_{v \in \cV}$ and in the case of a graph with no edges, it is the free convolution.  Moreover, C{\'e}bron, Oliveira Santos, and Youssef \cite{COSY2024} have established central limit theorems for graph products which directly generalize the central limit theorems for classical and free convolution.
				
				Hence, we view the study of atoms for a graph convolution as a natural generalization of the free and classical cases.  The atoms of a free convolution have been studied in depth \cite{BV1998,MaiSpeicherWeber2017,BelBerLiu2021} using the complex-analytic tools of the Cauchy transform and $R$-transform.  We do not yet have an analytically tractable method for exact computation of graph convolutions in general, and the study of its atoms seems like an important subproblem that could give hints about its behavior.
				
				For the free product case, the atoms of the graph convolution are closely related to the atomic summands of the free product von Neumann algebra.  Indeed, by \cite[Theorem 7.4]{BV1998} an atom for the free convolution can only occur at a point $\lambda$ if there are atoms of $\mu_v$ at some points $\lambda_v$ such that $\sum_{v \in \cV} \lambda_v = \lambda$ and the associated eigenspace projections $p_v$ have nonzero intersection (which is the $\lambda$-eigenspace projection for the sum).  In particular, the free sum has atoms if and only if the free product von Neumann algebra $*_{v \in \cV} (\mathrm{W}^*(x_v),\varphi_v)$ has atoms.  However, this last statement fails in general for graph products.  It can happen that the operator $\sum_{v \in \cV} x_v$ has an eigenvalue even though the algebra $\gp_{v \in \cG} (\mathrm{W}^*(x_v),\varphi_v)$ is diffuse, as the following example shows.
				
				\begin{example}
					Consider the graph $\cG$ with $\cV = \{1,2,3,4\}$ and two edges, connecting $1$ to $2$ and $3$ to $4$.
					Thus, for von Neumann algebras $M_1$, $M_2$, $M_3$, $M_4$ we have
					\[
					\gp_{j \in \cG} M_j = (M_1 \otimes M_2) * (M_3 \otimes M_4).
					\]
					Take $M_j$ to be generated by a single projection $p_j$ with state $\varphi_j$ given by
					\[
					\varphi_1(p_1) = \varphi_2(1 - p_2) = \varphi_3(p_3) = \varphi_4(1 - p_4) = 2^{-1/2}.
					\]
					We claim that $p_1 + \dots + p_4$ has an atom at $2$, even though $M = \mathrm{W}^*(p_1,\dots,p_4)$ is diffuse.  Note that $p_1 + p_2$ has an atom at $1$ of size
					\[
					(2^{-1/2})^2 + (1 - 2^{-1/2})^2 = \frac{1}{2} + \frac{3 - 2\sqrt{2}}{2} = 2 - \sqrt{2}.
					\]
					Likewise $p_3 + p_4$ has an atom of size $2 - \sqrt{2}$ at $1$.  Since $p_1 + p_2$ and $p_3 + p_4$ are freely independent, we know that $(p_1 + p_2) + (p_3 + p_4)$ has an atom at $2$ of size
					\[
					(2 - \sqrt{2}) + (2 - \sqrt{2}) - 1  = 3 - 2 \sqrt{2}.
					\]
					However, note that $M_1 \otimes M_2$ has atoms of size $1/2$, $(\sqrt{2}-1)/2$, $(\sqrt{2}-1)/2$, and $(3 - 2 \sqrt{2})/2$, and likewise $M_3 \otimes M_4$.  Since all the atoms in the algebras have size $\leq 1/2$, the free product $(M_1 \otimes M_2) * (M_3 \otimes M_4)$ is diffuse by \cite[Theorem 2.3]{Dykema1993freedimension}.
				\end{example}
				
				Therefore, our main result does not immediately give a full classification of the atoms for a graph convolution, which is another question we leave for future research.

				\subsection*{Organization}
				
				In \S \ref{sec: preliminaries} we recall the necessary background on von Neumann algebras and graph products.
				
				In \S \ref{sec: classification of atoms}, we classify the atoms in graph products in terms of the power series \eqref{eq: power series f}.  The argument is organized into three stages: (1) constructing the matrix units and showing that (if nonzero) they provide a central summand in the graph product, (2) computation of the ranges of the matrix units on the Hilbert space and consequently evaluation of $\varphi$ on these matrix units, and (3) proof that all finite-dimensional summands are realized by this construction.
				
				In \S \ref{sec: power series}, we evaluate the power series using the results of Cartier and Foata \cite{CF1969}.  We then prove the criteria for the existence of atoms stated in the main theorems and discuss the geometry of the region of convergence.
				
				\subsection*{Acknowledgements}
				
				Jekel was supported by a Marie Sk{\l}odowska-Curie Action from the European Union (FREEINFOGEOM, grant id: 101209517).  The authors thank Joachim Kock for pointing out the work of Cartier and Foata \cite{CF1969} during a visit of Charlesworth to Copenhagen, which was supported by the Marie Curie grant and funds from Cardiff University.  We thank Mario Klisse for pointing out the connection with works on Hecke algebras.  We thank the referees for their careful reading, corrections, and suggestions.

				\section{Preliminaries} \label{sec: preliminaries}
				
				We assume familiarity with von Neumann algebras with faithful normal states; for background, see e.g. \cite{Sakai1971,JonesSunder1997}. In particular, we assume standard facts about the GNS construction or standard representation $L^2(M,\varphi)$ and the modular automorphism group $\sigma_t^\varphi$ (though readers who are only interested in the tracial case may safely assume all the states are tracial and disregard mentions of the modular group).
				When it is helpful to draw a distinction between an element of $M$ and its corresponding vector in the GNS space $L^2(M, \varphi)$, we will write $x$ for the former and $\widehat{x}$ for the latter; when the clarification is unnecessary and the surrounding notation is dense we will omit the hat.
				
				We will denote by $\bN$ the set of positive integers, and by $\bN_0$ the set of non-negative integers.
				
				\subsection{Graphs} \label{subsec: graphs}
				
				Graphs in this paper are assumed to be finite, undirected, and simple, and hence for a graph $\cG = (\cV,\cE)$ we view the edge set $\cE$ as a symmetric subset of $\cV \times \cV$ disjoint from the diagonal.  We write $v \sim v'$ to mean that two vertices $v$ and $v'$ are adjacent.
				
				For a graph $\cG$ and $\cV' \subseteq \cV$, the \emph{subgraph induced by $\cV'$} is the graph $\cG'$ with vertex set $\cV '$ and edge set $\cE ' = \cE  \cap (\cV ' \times \cV ')$.
				
				For graphs $\cG_1$, \dots, $\cG_k$, the \emph{graph join} $\cG = \cG_1 + \dots + \cG_k$ is the graph with vertex set $\cV _1 \sqcup \dots \sqcup \cV _k$ and edge set $\cE$ given by
				\[ \cE = \cE_1 \sqcup \cdots \sqcup \cE_k \sqcup \bigsqcup_{i\neq j} \cV_i \times \cV_j. \]
				That is, vertices $v, v' \in \cV$ are adjacent if and only if:
				\begin{itemize}
					\item $v, v'$ are in the same $\cG_j$ and $(v,v') \in \cE_j$; or
					\item $v \in \cG_i$ and $v' \in \cG_j$ with $i \neq j$.
				\end{itemize}
				If $\cG$ cannot be expressed as the join of two (nonempty) graphs, then $\cG$ is said to be \emph{join-irreducible}.
				Every graph $\cG$ can be decomposed in a unique way as the join of some join-irreducible graphs, which correspond to the connected components of its complement.  
				
				\subsection{Words} \label{subsec: words}
				
				We recall the following terminology for words associated to a graph $\cG = (\cV,\cE)$.
				\begin{itemize}
					\item A \emph{word} on the alphabet $\cV$ is a string $w = w_1 \dots w_{\ell}$ where $w_i \in \cV$; $\ell$ is said to be the \emph{length} of the word.
					Note that the empty word $\emptyset$ is a valid word of length zero.
					\item A word $w$ is \emph{$\cG$-reduced} if whenever $i < j$ and $w_i = w_j$, then there exists $k$ between $i$ and $j$ such that $w_k \neq w_i$ and $w_k \not \sim w_i$.
					\item For a word $w$, an \emph{admissible swap} is a transformation of $w$ into another word $w'$ by switching two consecutive letters $w_i$ and $w_{i+1}$ such that $w_i \sim w_{i+1}$.
					\item Two words $w$ and $w'$ are \emph{equivalent} if $w$ can be transformed into $w'$ by a sequence of admissible swaps.
				\end{itemize}
				It is easy to verify that admissible swaps preserve the length of words and the property of being $\cG$-reduced.
				Moreover, note that admissible swaps never exchange two copies of the same letter; this will be important below when we consider tensor products corresponding to words $w$. 
				
				Throughout the paper, we fix a set $\mathcal{W}$ of representatives for the equivalence classes of reduced words, and write $\mathcal{W}_\ell$ for the elements of $\mathcal{W}$ with length $\ell$.
				In Section~\ref{sec: power series}, we will also need a set $\overline\cW$ of representatives for the equivalence classes of all words, including the ones which are not reduced.
				
				\subsection{Graph products}
				\label{subsec: graph products}
				
				To describe the definition of graph products of von Neumann algebras with faithful normal states, we first recall the graph product of Hilbert spaces with specified unit vectors.
				This construction is as in \cite[\S 4]{Mlot2004} and as in \cite[\S 3]{CaFi2017}.
				
				Let $\cG = (\cV,\cE)$ be a graph and for each $v \in \cV$, let $(H_v,\Omega_v)$ be a Hilbert space with a specified unit vector.
				Let $H_v^\circ = \Omega_v^{\perp} \subseteq H_v$, so that $H_v \cong \C \Omega_v \oplus H_v^\circ$.
				The \emph{graph product Hilbert space} $\gp_{v \in \cG} (H_v,\Omega_v)$ is the pair $(H,\Omega)$ where
				\[
				H = \C \Omega \oplus \bigoplus_{\ell \in \bN} \bigoplus_{w \in \mathcal{W}_\ell} H_{w_1}^\circ \otimes \dots \otimes H_{w_{\ell}}^\circ.
				\]
				We regard $\C \Omega$ as corresponding to the empty word.
				
				If $w = w_1 \dots w_{\ell}$ is an arbitrary word, then we write $H_w^\circ = H_{w_1}^\circ \otimes \dots \otimes H_{w_{\ell}}^\circ$, or $H_w^\circ = \bC\Omega$ if $w$ is the empty word (viewing $\bC\Omega$ as an empty tensor product).
				Note that if $w$ and $w'$ are equivalent $\cG$-reduced words of length $\ell$, then there is an associated isomorphism $H_w^\circ \to H_{w'}^\circ$, which for each $v$ sends the $j$th occurrence of the tensorand $H_v^\circ$ in $H_w^{\circ}$ to the $j$th occurrence of $H_v^\circ$ in $H_{w'}^{\circ}$.
				This latter fact follows since a sequence of admissible swaps never swaps two copies of the same letter.
				
				For each vertex $v$, the words may be partitioned into two sets: words $w$ such that the concatenation $vw$ is $\cG$-reduced, and words $w$ that are equivalent to a word $w'$ that begins with $v$.
				Hence, the equivalence classes of words can be enumerated by listing $w$ and $vw$ for every $w$ such that $vw$ is $\cG$-reduced, up to equivalence.
				Therefore,
				\[
				H \cong \bigoplus_{\substack{w \in \mathcal{W}\\ vw\ \cG\text{-reduced}}} (H_w^\circ \oplus H_{vw}^\circ) \cong \bigoplus_{\substack{w \in \mathcal{W}\\ vw\ \cG\text{-reduced}}} (\C \oplus H_v^\circ) \otimes H_w^\circ \cong H_v \otimes \bigoplus_{\substack{w \in \mathcal{W}\\ vw\ \cG\text{-reduced}}} H_w^\circ.
				\]
				Let $U_v$ be the unitary isomorphism from $H$ to the tensor product on the right-hand side.  We therefore obtain a $*$-homomorphism $\lambda_v: B(H_v) \to B(H)$ given by $\lambda_v(T) = U_v^*(T \otimes 1)U_v$.
				
				Given a graph $\cG = (\cV,\cE)$ and von Neumann algebras with faithful normal states $(M_v,\varphi_v)$ for $v \in \cV$, the  \emph{graph product von Neumann algebra} $(M,\varphi) = \gp_{v \in \cG} (M_v,\varphi_v)$ is defined as follows.  Take $H_v = L^2(M_v,\varphi_v)$ and $\Omega_v = \widehat{1} \in L^2(M_v,\varphi_v)$.  Form the graph product Hilbert space $(H,\Omega)$.  Then let $M$ be the von Neumann algebra generated by $\lambda_v(M_v)$ for $v \in \cV$, and let $\varphi$ be the state given by $\varphi(x) = \ip{\Omega, x\Omega}$.  It follows from \cite[\S 3.3]{CaFi2017} that $\varphi$ is a faithful state on $M$, and the map $\lambda_v: M_v \to M$ is state-preserving by construction.  Moreover, the inclusion $\lambda_v$ satisfies $\lambda_v \circ \sigma_t^{\varphi_v} = \sigma_t^{\varphi} \circ \lambda_v$ and equivalently there is a normal state-preserving conditional expectation $M \to M_v$.
				
				If $(M,\varphi)$ is the graph product, we may regard $M_v$ as a von Neumann subalgebra of $M$.  Then the subalgebras $(M_v)_{v \in \cV}$ satisfy the following notion of graph independence:  Given a $\cG$-reduced word $w$ and elements $x_j \in M_{w_j}$ with $\varphi_{w_j}(x_j) = 0$, we have $\varphi(x_1 \dots x_\ell) = 0$.  Note that for $w \in \mathcal{W}$ and $x_j \in M_{w_j}$ with $\varphi_{w_j}(x_j) = 0$, we have
				\[
				x_1 \dots x_\ell \Omega = \widehat{x}_1 \otimes \dots \otimes \widehat{x}_\ell.
				\]
				In particular, this shows that the action of the $M_v$'s on $\Omega$ generates the entire Hilbert space, and therefore the GNS space $(L^2(M,\varphi),\widehat{1})$ is isomorphic to $(H,\Omega)$ by sending $\widehat{x}$ to $x\Omega$.  The graph product von Neumann algebra $(M,\varphi)$ with the specified inclusions of $M_v$ into $M$ is characterized up to isomorphism by this notion of independence \cite[Proposition 3.22]{CaFi2017}.
				
				\section{Classification of atoms} \label{sec: classification of atoms}
				
				\subsection{Construction of matrix units} \label{subsec: construction of matrix units}
				
				\begin{notation} \label{nota: fd summands in gp}
					As usual we will denote by $[k]$ the set $\set{1, \ldots, k}$, and further write $[\vec n]$ for $\prod_{v \in \cV} [n(v)]$.
					We will use $\operatorname{swap}_v^a(\vec{\imath})$ to denote the vector obtained by replacing the $v$th entry of $\vec{\imath}$ with $a$.
					It is often useful to fix a concrete element of $[\vec n]$, so we will write $\vec1 \in [\vec n]$ for the vector with all entries $1$.
					
					Consider a graph product $(M,\varphi) = \gp_{v \in \cG} (M_v,\varphi_v)$ of von Neumann algebras with faithful normal states, and let $\iota_v: M_v \to M$ be the corresponding inclusion.  For each $v$, assume that $(M_v,\varphi_v)$ has a direct summand $(\mathbb{M}_{n(v)},\psi_v)$ with weight $\alpha^{(v)}$.  Let $(e_{i,j}^{(v)})_{i,j=1}^{n(v)}$ be the matrix units in this copy of $\mathbb{M}_{n(v)}$.  Assume that the state $\varphi_v$ on the matrix units satisfies
					\[
					\varphi_v(e_{i,j}^{(v)}) = \alpha^{(v)} \psi_v(e_{i,j}^{(v)}) = \mathbf{1}_{i=j} t_i^{(v)},
					\]
					or equivalently, $\alpha^{(v)} \psi_v$ is the trace against the diagonal matrix $(t_1^{(v)},\dots,t_{n(v)}^{(v)})$; note $t_j^{(v)} > 0$ by faithfulness of $\varphi_v$.
					For $\vec{\imath} \in [\vec n]$ let
					\[
					p_{\vec{\imath}} = \bigwedge_{v \in \cV} e_{i(v),i(v)}^{(v)}.
					\]
					
				\end{notation}
				
				\begin{lemma} \label{lemma: nonzero p means clique}
					With the setup of Notation \ref{nota: fd summands in gp}, suppose $p_{\vec{\imath}}$ is nonzero for all $\vec{\imath}$. Then $\cK = \{v \in \cV: n(v) > 1\}$ is a clique in $\cG$.
				\end{lemma}
				
				\begin{proof}
					We proceed by contrapositive.
					Suppose that $\cK$ is not a clique.
					Then there are $v$ and $v'$ such that $n(v) > 1$ and $n(v') > 1$ and $v$ is not adjacent to $v'$.
					Since $\sum_{i=1}^{n(v)} \varphi(e_{i,i}^{(v)}) \leq 1$, there exists some $i$ with $\varphi(e_{i,i}^{(v)}) \leq 1/2$.
					Similarly, there exists $i'$ with $\varphi(e_{i',i'}^{(v')}) \leq 1/2$.
					Since $e_{i,i}^{(v)}$ and $e_{i',i'}^{(v')}$ are freely independent, $e_{i,i}^{(v)} \wedge e_{i',i'}^{(v')} = 0$; indeed, the intersection of the projections is the $2$-eigenspace projection of $e_{i,i}^{(v)} + e_{i',i'}^{(v')}$ and it follows from \cite[Theorem 7.4]{BV1998} that the spectral measure of $e_{i,i}^{(v)} + e_{i',i'}^{(v)}$ has no atom at $2$.  Let $\vec{\imath}$ be some tuple of indices with $\vec{\imath}(v) = i$ and $\vec{\imath}(v') = i'$.  Then $p_{\vec{\imath}} \leq e_{i,i}^{(v)} \wedge e_{i',i'}^{(v')} = 0$.
				\end{proof}
				
				For the rest of \S\ref{subsec: construction of matrix units} and \S \ref{subsec: Fock space}, we will assume that $\cK = \{v \in \cV: n(v) > 1\}$ is a clique.
				
				\begin{notation} \label{nota: fd summands in gp 2}
					Using the setup of Notation~\ref{nota: fd summands in gp}, for $\vec{\imath}$, $\vec{\jmath} \in [\vec n]$, observe that
					\[
					e_{\vec{\imath},\vec{\jmath}}^{(\cK)} = \prod_{v \in \cK} e_{i(v),j(v)}^{(v)}
					\]
					is a partial isometry (since the terms in the product commute), and its source projection is
					\[
					e_{\vec{\jmath},\vec{\jmath}}^{(\cK)} = \bigwedge_{v \in \cK} e_{j(v),j(v)}^{(v)} \geq p_{\vec{\jmath}}.
					\]
					Therefore, $e_{\vec{\imath},\vec{\jmath}}^{(\cK)} p_{\vec{\jmath}} e_{\vec{\jmath},\vec{\imath}}^{(\cK)}$ is a projection, and so the following projections are well-defined:
					\[
					q_{\vec{\imath}} = \bigwedge_{\vec{j} \in [\vec n]} e_{\vec{\imath},\vec{\jmath}}^{(\cK)} p_{\vec{\jmath}} e_{\vec{\jmath},\vec{\imath}}^{(\cK)}.
					\]
					Note that by taking $\vec{\jmath} = \vec{\imath}$ as a candidate in the meet, we see $q_{\vec{\imath}} \leq p_{\vec{\imath}}$.  We furthermore note that
					\[
					q_{\vec{\imath}} = e_{\vec{\imath},\vec{\jmath}}^{(\cK)} q_{\vec{\jmath}} e_{\vec{\jmath},\vec{\imath}}^{(\cK)}.
					\]
				\end{notation}
				
				It may happen that $q_{\vec{\imath}} = 0$ for some $\vec{\imath}$, which also implies that $q_{\vec{\imath}} = 0$ for all $\vec{\imath}$.
				We will determine conditions where it is zero in the next subsection (see Lemma~\ref{lem: Hilbert space computation of projections}).  For now, we want to show that if they are nonzero, these projections extend to a set of matrix units in a direct summand of $M$.
				
				\begin{lemma} \label{lem: construction of f d summand and state}
					Suppose $q_{\vec{\imath}} \neq 0$.
					For $\vec{\imath}$, $\vec{\jmath} \in [\vec n]$, let
					\[
					f_{\vec{\imath},\vec{\jmath}} = e_{\vec{\imath},\vec{k}}^{(\cK)} q_{\vec{k}} e_{\vec{k},\vec{\jmath}}^{(\cK)},
					\]
					which is independent of the choice of $\vec{k}$.  Let
					\[
					N = \Span\set*{f_{\vec{\imath},\vec{\jmath}}: \vec{\imath}, \vec{\jmath} \in [\vec n] }.
					\]
					Then
					\[
					(N,\psi) \cong \bigotimes_{v \in \cV} (\mathbb{M}_{n(v)},\psi_v) \cong \bigotimes_{v \in \cK} (\mathbb{M}_{n(v)},\psi_v),
					\]
					where $\psi = (1/\alpha)\varphi|_N$ and $\alpha = \varphi(1_N)$.
					(Here the second isomorphism comes from the fact that the terms corresponding to $v \in \cV \setminus \cK$ are simply $\mathbb{C}$, so they can be ignored when convenient).
					Under the isomorphism above, $f_{\vec{\imath},\vec{\jmath}} \in N$ corresponds to $e_{\vec{\imath},\vec{\jmath}}^{(\cK)}$ on the right-hand side.  
				\end{lemma}
				
				\begin{proof}
					Note that
					\[
					f_{\vec{\imath},\vec{\jmath}} f_{\vec{\jmath},\vec{k}} = e_{\vec{\imath},\vec{1}}^{(\cK)} q_{\vec{1}} e_{\vec{1},\vec{\jmath}}^{(\cK)} e_{\vec{\jmath},\vec{1}}^{(\cK)} q_{\vec{1}} e_{\vec{1},\vec{\jmath}}^{(\cK)} =  e_{\vec{\imath},\vec{1}}^{(\cK)} q_{\vec{1}} e_{\vec{1},\vec{1}}^{(\cK)} q_{\vec{1}} e_{\vec{1},\vec{\jmath}}^{(\cK)} = e_{\vec{\imath},\vec{1}}^{(\cK)} q_{\vec{1}} e_{\vec{1},\vec{\jmath}}^{(\cK)}.
					\]
					Similarly, $f_{\vec{\imath},\vec{\jmath}}^* = f_{\vec{\jmath},\vec{\imath}}$.  Thus, the generators satisfy the relations of matrix units, so we have the claimed isomorphism for $N$ as a $*$-algebra.
					
					Now we must show that the isomorphism preserves the state.  By our assumption on $\varphi_v|_{\mathbb{M}_{n(v)}}$, we have for $x \in M_v$ that
					\[
					\varphi_v(e_{i,j}^{(v)} x) = \frac{t_i^{(v)}}{t_j^{(v)}} \varphi_v(x e_{i,j}^{(v)}).
					\]
					Since there is a state-preserving conditional expectation $E_v: M \to M_v$, we have that for all $x \in M$,
					\[
					\varphi(e_{i,j}^{(v)} x) = \varphi_v(e_{i,j}^{(v)} E_v[x]) = \frac{t_i^{(v)}}{t_j^{(v)}} \varphi_v(E_v[x] e_{i,j}^{(v)}) = \frac{t_i^{(v)}}{t_j^{(v)}} \varphi(x e_{i,j}^{(v)}).
					\]
					Since the operators $e_{i,j}^{(v)}$ and $e_{i',j'}^{(v')}$ commute for distinct $v, v' \in \cK$, we see that
					\[
					\varphi(e_{\vec{\imath},\vec{\jmath}}^{(\cK)} x) = \prod_{v \in \cK} \frac{t_{i(v)}^{(v)}}{t_{j(v)}^{(v)}} \varphi(x e_{\vec{\imath},\vec{\jmath}}^{(\cK)}).
					\]
					In particular, when $\vec{\imath} \neq \vec{\jmath}$, we have
					\[
					\varphi(f_{\vec{\imath},\vec{\jmath}})
					= \varphi(e_{\vec{\imath},\vec{1}}^{(\cK)} q_{\vec{1}} e_{\vec{1},\vec{\jmath}}^{(\cK)})
					= \prod_{v \in \cK} \frac{t_{i(v)}^{(v)}}{t_{1}^{(v)}} \varphi(q_{\vec{1}} e_{\vec{1},\vec{\jmath}}^{(\cK)} e_{\vec{\imath},\vec{1}}^{(\cK)})
					= 0.
					\]
					Thus,
					\[
					\psi(f_{\vec{\imath},\vec{\jmath}}) = 0 = \left[ \bigotimes_{v \in \cK} \psi_v \right](e_{\vec{\imath},\vec{\jmath}}^{(\cK)}).
					\]
					Now note that since $\varphi|_N$ is proportional to $\psi$
					\[
					\frac{\psi(f_{\vec{\imath},\vec{\imath}})}{\psi(f_{\vec{\jmath},\vec{\jmath}})}
					= \frac{\varphi(e_{\vec{\imath},\vec{\jmath}}^{(\cK)} f_{\vec{\jmath},\vec{\imath}})}{\varphi(f_{\vec{\jmath},\vec{\imath}} e_{\vec{\imath},\vec{\jmath}}^{(\cK)})}
					= \prod_{v \in \cK} \frac{t_{i(v)}^{(v)}}{t_{j(v)}^{(v)}}
					= \frac{[\otimes_{v \in \cK} \psi_v](e_{\vec{\imath},\vec{\imath}}^{(\cK)})}{[\otimes_{v \in \cK} \psi_v](e_{\vec{\jmath},\vec{\jmath}}^{(\cK)})}.
					\]
					Thus, the values of $\psi$ and the values of $\bigotimes_{v \in \cK} \psi_v$ on the corresponding projections in the respective algebras are proportional.  Hence, the values on all the matrix units are proportional since the states vanish on the off-diagonal matrix units.  Because the states are normalized, we conclude that $(N,\psi)$ is isomorphic to $\bigotimes_{v \in \cK} (\mathbb{M}_{n(v)},\psi_v)$.
				\end{proof}
				
				\begin{lemma}
					Suppose $q_{\vec{\imath}} \neq 0$.  Then $N$ is a finite-dimensional two-sided ideal in $M$, and hence a direct summand.  Let $\pi_v: M_v \to \mathbb{M}_{n(v)}$ and $\pi: M \to N$ be the quotient maps associated to the direct sum decompositions, and let
					\[
					\kappa_v: \mathbb{M}_{n(v)} \to \bigotimes_{v \in \cV} \mathbb{M}_{n(v)} \cong N
					\]
					be the map which tensors its input with $\bigotimes_{v' \in \cV\setminus\set v}1_{\bM_{n(v')}}$, and then applies the isomorphism from Lemma~\ref{lem: construction of f d summand and state}.
					Then $\pi \circ \iota_v = \kappa_v \circ \pi_v$.
				\end{lemma}
				
				\begin{proof}
					First, we show that $N$ is a two-sided ideal.  Since $N$ is closed under adjoints it suffices to show it is a left ideal.  It further suffices to show that it is a left $M_v$-module for each $v \in \cV$.  Let $v \in \cV$ and $x \in M_v$.  Let $1_{\mathbb{M}_{n(v)}}$ be the unit in $\mathbb{M}_{n(v)}$.  Thus,
					\[
					x = x(1 - 1_{\mathbb{M}_{n(v)}}) + \pi_v(x) 1_{\mathbb{M}_{n(v)}},
					\]
					where we view $\pi_v(x) \in \mathbb{M}_{n(v)} \subseteq M_v$ non-unitally.  Take one of the matrix units $f_{\vec{\imath},\vec{\jmath}}$ from $N$.  Its range projection is $q_{\vec{\imath}} \leq p_{\vec{\imath}} \leq 1_{\mathbb{M}_{n(v)}}$. Thus,
					\[
					x f_{\vec{\imath},\vec{\jmath}} = 0 + \pi_v(x) f_{\vec{\imath},\vec{\jmath}}.
					\]
					In the case where $v \in \cV \setminus \cK$, we are already done because $\pi_v(x)$ is a scalar multiple of $1_{\mathbb{M}_{n(v)}}$ and so acts by the identity on $f_{\vec{\imath},\vec{\jmath}}$.  Otherwise, suppose that $v \in \cK$, and write
					\[
					\pi_v(x) = \sum_{a,b \in [n(v)]} \gamma_{a,b} e_{a,b}^{(v)}.
					\]
					Then
					\begin{equation} \label{eq: left module computation}
						x f_{\vec{\imath},\vec{\jmath}} = \sum_{a,b \in [n(v)]} \gamma_{a,b} e_{a,b}^{(v)} e_{\vec{\imath},\vec{1}}^{(\cK)} q_{\vec{1}} e_{\vec{1},\vec{\jmath}}^{(\cK)} = \sum_{a,b \in [n(v)]} \gamma_{a,b} \mathbf{1}_{b=i(v)} e_{\operatorname{swap}_v^a(\vec{\imath}),\vec{1}}^{(\cK)} q_{\vec{1}} e_{\vec{1},\vec{\jmath}}^{(\cK)} \in N.
					\end{equation}
					Hence, $N$ is a left ideal as desired.
					
					Next, we show that $\pi \circ \iota_v = \kappa_v \circ \pi_v$ using the appropriate left $M_v$-module structures on each of the algebras in question.  Of course, $M_v$ is viewed as a left module over itself, and $M$ is viewed as a left $M_v$-module using the algebra inclusion $\iota_v$.  Moreover, since $N$ is a direct summand and hence an ideal in $M$, it is in particular a left $M_v$-submodule, and the projection $\pi: M \to N$ is a left $M_v$-module map, and so $\pi \circ \iota_v$ is a left $M_v$-module map.  Now turning our attention to $\kappa_v \circ \pi_v$, note that $N_v$ is an ideal in $M_v$ and the projection $\pi_v$ is an $M_v$-module map.  The computation \eqref{eq: left module computation} shows that $\kappa_v$ is a left $M_v$-module map, and so $\kappa_v \circ \pi_v$ is a left $M_v$-module map.  Overall, the maps $\pi \circ \iota_v$ and $\kappa_v \circ \pi_v$ are both left $M_v$-module maps and they agree at $1$, hence they are equal.
				\end{proof}
				
				\subsection{Fock space computations} \label{subsec: Fock space}
				
				In this subsection, we continue with the same setup as \S \ref{subsec: construction of matrix units}, and in particular, we retain the assumption $\cK = \{v \in \cV: n(v) > 1\}$ is a clique.  In the next lemma, we give an explicit description of the range of $q_{\vec{\imath}}$ in terms of the GNS Hilbert space associated to the graph product.  It turns out that we need to study $\prod_{v \in \cK} n(v)$ many vectors simultaneously, and we introduce an auxiliary copy of $\bigotimes_{v \in \cV} \mathbb{M}_{n(v)}$ to assist with book-keeping in the statements.
				
				\begin{notation}
					For $v \in \cV$, let
					\[
					\varepsilon_{i,j}^{(v)} = e_{i,j}^{(v)} \otimes \bigotimes_{v' \in \cV \setminus \{v\}} 1 \in \bigotimes_{v \in \cV} \mathbb{M}_{n(v)}.
					\]
					Moreover, for $\vec{\imath} \in [\vec n]$, let
					\[
					\delta_{\vec{\imath}} = \bigotimes_{v \in \cV} \delta_{i(v)} \in \bigotimes_{v \in \cV} \mathbb{C}^{n(v)},
					\]
					where $\delta_i$ denotes the standard basis vector.
				\end{notation}
				
				\begin{lemma}
					\label{lem: master vector lemma}
					Let $H_v = L^2(M_v,\varphi_v)$ and $H$ be the graph product Hilbert space, which we may identify with $L^2(M,\varphi)$.
					Let $\xi \in \Ran(q_{\vec{1}})$, and set
					\[
					\xi_{\vec{\imath}} = e_{\vec{\imath},\vec{1}}^{(\cK)} \xi.
					\]
					We have
					\begin{equation} \label{eq: master vector setup}
						\xi_{\vec{\imath}} \in \Ran(q_{\vec{\imath}}), \qquad e_{\vec{\imath},\vec{\jmath}}^{(\cK)} \xi_{\vec{\jmath}} = \xi_{\vec{\imath}}.
					\end{equation}
					Let $\alpha_{\vec{\imath}} = \ip{\Omega,\xi_{\vec{\imath}}}$.  Let
					\[
					\mathring{e}_{i,j}^{(v)} = e_{i,j}^{(v)} - \varphi(e_{i,j}^{(v)}) 1 = e_{i,j}^{(v)} - \mathbf{1}_{i=j} t_i^{(v)} 1.
					\]
					Then viewing each $\mathring{e}_{a,b}^{(w)}$ as a vector in $L^2(M_w)$ (and omitting the hat) we have
					\begin{equation} \label{eq: master vector formula}
						\xi_{\vec{\imath}} =
						\sum_{\vec{\jmath} \in [\vec n]} \alpha_{\vec{\jmath}} \sum_{\ell \in \bN_0} \sum_{w \in \mathcal{W}_\ell} \sum_{\substack{a_1,b_1 \in [n(w_1)] \\ \vdots \\ a_{\ell}, b_{\ell} \in [n(w_\ell)]}} \ip{\delta_{\vec{\imath}}, \varepsilon_{a_1,b_1}^{(w_1)} \dots \varepsilon_{a_\ell,b_\ell}^{(w_\ell)} \delta_{\vec{\jmath}}} \, \frac{1}{t_{b_1}^{(w_1)} \dots t_{b_\ell}^{(w_\ell)}} \, \mathring{e}_{a_1,b_1}^{(w_1)} \otimes \dots \otimes \mathring{e}_{a_\ell,b_\ell}^{(w_\ell)}.
					\end{equation}
					When $\ell = 0$ and $w$ is the empty word, we interpret the sum as containing a single term $\ang{\delta_{\vec\imath}, \delta_{\vec\jmath}}\Omega$, corresponding to the unique choice of $0$-tuples for $a$ and $b$.  The set of terms appearing in the sum in \eqref{eq: master vector formula}, namely
					\[
					\set*{
						\ip{\delta_{\vec{\imath}}, \varepsilon_{a_1,b_1}^{(w_1)} \dots \varepsilon_{a_\ell,b_\ell}^{(w_\ell)} \delta_{\vec{\jmath}}} \,
						\frac{1}{t_{b_1}^{(w_1)} \dots t_{b_\ell}^{(w_\ell)}} \,
						\mathring{e}_{a_1,b_1}^{(w_1)} \otimes \dots \otimes \mathring{e}_{a_\ell,b_\ell}^{(w_\ell)}
						: \vec\jmath\in[\vec n]; \ell\in\bN_0; w\in\cW_\ell; a_1, b_1 \in [n(w_1)]; \ldots; a_\ell, b_\ell \in [n(w_\ell)]
					},
					\]
					is orthogonal.
					
					Conversely, given coefficients $\alpha_{\vec \imath}$, if the sums \eqref{eq: master vector formula} converge, then the vectors $\xi_{\vec \imath}$ given by \eqref{eq: master vector formula} satisfy \eqref{eq: master vector setup}.
					
				\end{lemma}
				
				\begin{proof}
					We adopt the notation related to the graph product Hilbert space $H$ described in \S\ref{subsec: graph products}.
					Write
					\[
					\xi_{\vec{\imath}} = \bigoplus_{w \in \mathcal{W}} \xi_{\vec{\imath},w}, \qquad \xi_{\vec{\imath},w} \in H_{w}^\circ.
					\]
					For convenience, if a $\cG$-reduced word $w'$ is equivalent to a word $w \in \mathcal{W}$, we write $\xi_{\vec\imath,w'}$ for the vector $\xi_{\vec{\imath},w}$ with the tensorands permuted according to the rearrangement that transforms $w$ into $w'$.
					
					We want to establish the formula for $\xi_{\vec{\imath},w}$ by induction on the length of $w$ using the relations $e^{(\cK)}_{\vec{\imath},\vec{\jmath}} \xi_{\vec{\jmath}} = \xi_{\vec{\imath}}$.
					We first describe the actions of the matrix units $e_{a,b}^{(v)}$ on $\xi_{\vec{\jmath}}$.  First, suppose that $v \in \cK$. 
					Since
					\[
					\xi_{\vec{\jmath}} \in \Ran(q_{\vec{\jmath}}) \subseteq \Ran(e_{\vec{\jmath},\vec\jmath}^{(\cK)}),
					\]
					we have
					\[
					e_{a,b}^{(v)} \xi_{\vec{\jmath}} = e_{a,b}^{(v)} e_{\vec{\jmath},\vec\jmath}^{(\cK)} \xi_{\vec{\jmath}} = \mathbf{1}_{b=j(v)} e_{\operatorname{swap}_{v}^{a}(\vec{\jmath}),\vec\jmath}^{(\cK)} \xi_{\vec{\jmath}} = \mathbf{1}_{b=j(v)} \xi_{\operatorname{swap}_{v}^{a}(\vec{\jmath})}.
					\]
					On the other hand, for $v \in \cV\setminus\cK$ we have $n(v) = 1$ so $a=b=1$ and $\xi_{\vec \jmath} \in \Ran(q_{\vec\jmath})\subseteq \Ran(p_{\vec\jmath}) \subseteq \Ran(e_{1,1}^{(v)})$ so
					\[e_{1,1}^{(v)} \xi_{\vec \jmath} = \xi_{\vec \jmath} = \mathbf{1}_{1 = j(v)}\xi_{\operatorname{swap}_v^1(\vec\jmath)}.\]
					Hence for all $v \in \cV$ we have
					\begin{equation} \label{eq: elementary action}
						e_{a,b}^{(v)} \xi_{\vec{\jmath}} = \mathbf{1}_{b=j(v)} \xi_{\operatorname{swap}_{v}^{a}(\vec{\jmath})}.
					\end{equation}
					
					Now, for each $w \in \cW$ for which $vw$ is reduced, the unitary isomorphism $U_v$ from \S\ref{subsec: graph products} carries $H_w^\circ \oplus H_{vw}^\circ$ onto $H_v \otimes H_w^\circ$, and the action of $e_{a,b}^{(v)}$ on $H$ corresponds to that of $e_{a,b}^{(v)} \otimes 1$.
					In particular, each $H_v \otimes H_w^{\circ}$ is invariant under the action of each $e_{a,b}^{(v)}\otimes 1$, and so \eqref{eq: elementary action} yields
					\begin{equation}\label{eq: elementary action 2}
						(e_{a,b}^{(v)}\otimes 1)U_v(\xi_{\vec\imath, w} \oplus \xi_{\vec\imath, vw}) = \mathbf{1}_{b=i(v)}U_v(\xi_{\operatorname{swap}_v^a(\vec\imath), w} \oplus \xi_{\operatorname{swap}_v^a(\vec\imath), vw}).
					\end{equation}
					Moreover, specializing the above to the case $a = b = i(v)$ or simply noting that $\xi_{\vec\imath} \in \Ran(q_{\vec\imath}) \subseteq \Ran(e_{i(v), i(v)}^{(v)})$ and applying the same decomposition argument, we see
					\[
					U_v(\xi_{\vec\imath, w} \oplus \xi_{\vec\imath, vw}) \in \Ran(e_{i(v), i(v)}^{(v)}\otimes 1).
					\]
					
					
					Now $\Ran(e_{i(v),i(v)}^{(v)})$ in $H_v = L^2(M_v,\varphi_v)$ (understood as acting on the left) is spanned by $(e_{i(v),b}^{(v)})_{b=1}^{n(v)}$.
					Therefore, there are some vectors $\eta_{\vec{\imath},b,w,v} \in H_w^\circ$ such that
					\begin{equation} \label{eq: xi decomposition}
						U_v(\xi_{\vec{\imath},w} \oplus \xi_{\vec{\imath},vw}) = \sum_{b \in [n(v)]} e_{i(v),b}^{(v)} \otimes \eta_{\vec{\imath},b,w,v}.
					\end{equation}
					We rewrite this again in terms of $H_w^\circ \oplus H_{vw}^\circ$ by recalling
					\[
					e_{i(v),b}^{(v)} = \varphi(e_{i(v),b}^{(v)}) 1 + \mathring{e}_{i(v),b}^{(v)} \in \C \oplus H_v^\circ,
					\]
					and $\varphi(e_{i(v),b}^{(v)}) = \mathbf{1}_{i(v)=b} t_{i(v)}^{(v)}$.  Thus,
					\begin{equation} \label{eq: some sums equal}
						\xi_{\vec{\imath},w} \oplus \xi_{\vec{\imath},vw} = t_{i(v)}^{(v)} \eta_{\vec{\imath},i(v),w,v} \oplus \sum_{b \in [n(v)]} \mathring{e}_{i(v),b}^{(v)} \otimes \eta_{\vec{\imath},b,w,v},
					\end{equation}
					whence
					\begin{equation} \label{eq: first eta equation}
						\eta_{\vec{\imath},i(v),w,v} = \frac{1}{t_{i(v)}^{(v)}} \xi_{\vec{\imath},w}. 
					\end{equation}
					We now wish to evaluate the terms $\eta_{\vec{\imath},b,w,v}$ for $b \neq i(v)$.
					Take $a = i(v)$ and let $\vec{\jmath} = \operatorname{swap}_v^b(\vec\imath)$; note that
					\[
					\vec{\imath} = \operatorname{swap}_v^a(\vec{\jmath}).
					\]
					From \eqref{eq: elementary action 2} and \eqref{eq: xi decomposition}, we have
					\[
					\sum_{c \in [n(v)]} e_{i(v),c}^{(v)} \otimes \eta_{\vec{\imath},c,w,v} = (e_{a,b}^{(v)} \otimes 1) \sum_{c \in [n(v)]} e_{j(v),c}^{(v)} \otimes \eta_{\vec{\jmath},c,w,v} = \sum_{c \in [n(v)]} e_{i(v),c}^{(v)} \otimes \eta_{\vec{\jmath},c,w,v}.
					\]
					Therefore, $\eta_{\vec{\imath},c,w,v} = \eta_{\vec{\jmath},c,w,v}$, and in particular taking $k = b = j(v)$,
					\[
					\eta_{\vec{\imath},b,w,v} = \eta_{\vec{\jmath},j(v),w,v} = \frac{1}{t_{j(v)}^{(v)}} \xi_{\vec{\jmath},w} = \frac{1}{t_b^{(v)}} \xi_{\operatorname{swap}_v^b(\vec{\imath}),w}.
					\]
					Therefore, overall, substituting this back into \eqref{eq: some sums equal}, we get
					\begin{equation} \label{eq: individual term preparation}
						\xi_{\vec{\imath},vw} = \sum_{b \in [n(v)]}  \mathring{e}_{i(v),b}^{(v)} \otimes \frac{1}{t_b^{(v)}} \xi_{\operatorname{swap}_v^b(\vec{\imath}),w}.
					\end{equation}
					We deduce that
					\begin{equation} \label{eq: sum prepared for ind hyp}
						\xi_{\vec{\imath},vw} = \sum_{a, b \in [n(v)]} \sum_{\vec{\jmath} \in [\vec n]} \ip{\delta_{\vec{\imath}}, \varepsilon_{a,b}^{(v)} \delta_{\vec{\jmath}}} \frac{1}{t_b^{(v)}} \mathring{e}_{a,b}^{(v)} \otimes  \xi_{\vec{\jmath},w}
					\end{equation}
					because $\ip{\delta_{\vec\imath}, \varepsilon_{a,b}^{(v)} \delta_{\vec\jmath}}$ evaluates to the indicator function that $a = i(v)$ and $\vec\jmath = \operatorname{swap}_b^v(\vec{\imath})$.
					
					We are now ready to show by induction that for reduced words $w$ of length $\ell$,
					\begin{equation} \label{eq: master vector formula inductive}
						\xi_{\vec{\imath},w} = \sum_{\vec{\jmath} \in [\vec n]} \alpha_{\vec{\jmath}} \sum_{\substack{a_1,b_1 \in [n(w_1)] \\ \vdots \\ a_{\ell}, b_{\ell} \in [n(w_\ell)]}} \ip{\delta_{\vec{\imath}}, \varepsilon_{a_1,b_1}^{(w_1)} \dots \varepsilon_{a_\ell,b_\ell}^{(w_\ell)} \delta_{\vec{\jmath}}} \, \frac{1}{t_{b_1}^{(w_1)} \dots t_{b_\ell}^{(w_\ell)}} \, \mathring{e}_{a_1,b_1}^{(w_1)} \otimes \dots \otimes \mathring{e}_{a_\ell,b_\ell}^{(w_\ell)}.
					\end{equation}
					The base case is when $w$ is the empty word, and in this case the formula holds trivially by definition of $\alpha_{\vec{\imath}}$ thanks to our interpretation of the sum over no choices of $a$ or $b$ as $\ip{\delta_{\vec\imath},\delta_{\vec\jmath}}\Omega$.
					Now suppose the formula holds for a reduced word $w$ and that $v \in \cV$ such that $vw$ is reduced.  Then by \eqref{eq: sum prepared for ind hyp} and the induction hypothesis,
					\begin{align}
						\xi_{\vec{\imath},vw} &= \sum_{\vec{k} \in [\vec n]} \sum_{a, b \in [n(v)]} \ip{\delta_{\vec{\imath}}, \varepsilon_{a,b}^{(v)} \delta_{\vec{k}}} \frac{1}{t_b^{(v)}} \mathring{e}_{a,b}^{(v)} \otimes  \xi_{\vec{k},w} \label{eq: horrid vector computation} \\
						&= \sum_{\vec{k} \in [\vec n]} \sum_{a, b \in [n(v)]} \ip{\delta_{\vec{\imath}}, \varepsilon_{a,b}^{(v)} \delta_{\vec{k}}} \frac{1}{t_b^{(v)}} \mathring{e}_{a,b}^{(v)} \otimes \nonumber \\
						& \qquad \sum_{\vec{\jmath} \in [\vec n]} \alpha_{\vec{\jmath}} \sum_{\substack{a_1,b_1 \in [n(w_1)] \\ \vdots \\ a_{\ell}, b_{\ell} \in [n(w_\ell)]}} \ip{\delta_{\vec{k}}, \varepsilon_{a_1,b_1}^{(w_1)} \dots \varepsilon_{a_\ell,b_\ell}^{(w_\ell)} \delta_{\vec{\jmath}}} \, \frac{1}{t_{b_1}^{(w_1)} \dots t_{b_\ell}^{(w_\ell)}} \, \mathring{e}_{a_1,b_1}^{(w_1)} \otimes \dots \otimes \mathring{e}_{a_\ell,b_\ell}^{(w_\ell)} \nonumber \\
						&= \sum_{\vec{\jmath} \in [\vec n]} \alpha_{\vec{\jmath}} \sum_{a, b \in [n(v)]} \sum_{\substack{a_1,b_1 \in [n(w_1)] \\ \vdots \\ a_{\ell}, b_{\ell} \in [n(w_\ell)]}} \left( \sum_{\vec{k} \in [\vec n]} \ip{\delta_{\vec{\imath}}, \varepsilon_{a,b}^{(v)} \delta_{\vec{k}}} \ip{\delta_{\vec{k}}, \varepsilon_{a_1,b_1}^{(w_1)} \dots \varepsilon_{a_\ell,b_\ell}^{(w_\ell)} \delta_{\vec{\jmath}}} \right)  \nonumber \\
						& \qquad  \frac{1}{t_b^{(v)} t_{b_1}^{(w_1)} \dots t_{b_\ell}^{(w_\ell)}} \, \mathring{e}_{a,b}^{(v)} \otimes \mathring{e}_{a_1,b_1}^{(w_1)} \otimes \dots \otimes \mathring{e}_{a_\ell,b_\ell}^{(w_\ell)} \nonumber \\
						&= \sum_{\vec{\jmath} \in [\vec n]} \alpha_{\vec{\jmath}} \sum_{a, b \in [n(v)]} \sum_{\substack{a_1,b_1 \in [n(w_1)] \\ \vdots \\ a_{\ell}, b_{\ell} \in [n(w_\ell)]}} \left( \ip{\delta_{\vec{\imath}}, \varepsilon_{a,b}^{(v)} \varepsilon_{a_1,b_1}^{(w_1)} \dots \varepsilon_{a_\ell,b_\ell}^{(w_\ell)} \delta_{\vec{\jmath}}} \right)  \nonumber \\
						& \qquad  \frac{1}{t_b^{(v)} t_{b_1}^{(w_1)} \dots t_{b_\ell}^{(w_\ell)}} \, \mathring{e}_{a,b}^{(v)} \otimes \mathring{e}_{a_1,b_1}^{(w_1)} \otimes \dots \otimes \mathring{e}_{a_\ell,b_\ell}^{(w_\ell)}.
					\end{align}
					This shows the claim for $vw$ and hence completes the induction.

					Next, we show the claimed orthogonality.  Of course, terms associated to two different words in $\cW$ are orthogonal by construction of $H$.
					Now consider two distinct terms associated to the same word $w \in \cW_\ell$, and suppose they correspond to choices $\vec\jmath \in [\vec n]$, and $a_t, b_t \in [n(t)]$; and ${\vec\jmath\,}' \in [\vec n]$, and $a_t', b_t' \in [n(t)]$ respectively.
					If either term vanishes, there is nothing to prove, so we may assume the inner products weighting them are non-zero.
					Hence we have
					\[
					\delta_{\vec\jmath} = (\varepsilon_{a_{\ell},b_{\ell}}^{(w_{\ell})})^* \dots (\varepsilon_{a_1,b_1}^{(w_1)})^* \delta_{\vec{\imath}}, \qquad
					\delta_{{\vec\jmath\,}'} = (\varepsilon_{a_{\ell}',b_{\ell}'}^{(w_{\ell})})^* \dots (\varepsilon_{a_1',b_1'}^{(w_1)})^* \delta_{\vec{\imath}}
					\]
					In particular, if all $a_t = a_t'$ and all $b_t = b_t'$ then $\vec\jmath = {\vec\jmath\,}'$ as well; since the terms are distinct, this cannot happen.
					Thus there is some minimum $1 \leq m \leq \ell$ so that $(a_m, b_m) \neq (a_m', b_m')$, and so there is some $\vec{k}$ such that
					\[
					\delta_{\vec{k}} = (\varepsilon_{a_{m-1},b_{m-1}}^{(w_{m-1})})^* \dots (\varepsilon_{a_1,b_1}^{(w_1)})^* \delta_{\vec{\imath}} = (\varepsilon_{a_{m-1}',b_{m-1}'}^{(w_{m-1})})^* \dots (\varepsilon_{a_1',b_1'}^{(w_1)})^* \delta_{\vec{\imath}}.
					\]
					We then see that $(\varepsilon_{a_m,b_m}^{(w_m)})^* \delta_{\vec{k}} = 0$ unless $a_m = k(w_m)$, and similarly $(\varepsilon_{a_m,b_m}^{(w_m)})^* \delta_{\vec{k}} = 0$ unless $a_m' = k(w_m)$.
					Hence, for the terms to be nonzero, we need $a_m = a_m'$, so $b_m \neq b_m'$.
					In particular, one of $e_{a_m,b_m}^{(w_m)}$ or $e_{a_m',b_m'}^{(w_m)}$ must be an off-diagonal matrix unit, which is in the kernel of the state.
					Therefore, $\mathring{e}_{a_m,b_m}^{(w_m)}$ and $\mathring{e}_{a_m',b_m'}^{(w_m)}$ are orthogonal.
					Hence also $\mathring{e}_{a_1,b_1}^{(w_1)} \otimes \dots \otimes \mathring{e}_{a_\ell,b_\ell}^{(w_\ell)}$ and $\mathring{e}_{a_1',b_1'}^{(w_1)} \otimes \dots \otimes \mathring{e}_{a_\ell',b_\ell'}^{(w_\ell)}$ are orthogonal.
					
					Finally, we must prove the converse claim.  Since the terms in the sum are orthogonal, convergence of the sum in $L^2$ is independent of the order of summation.  Fix coefficients $\alpha_{\vec \imath}$ for $\vec \imath \in [\vec n]$, assume that the sums in \eqref{eq: master vector formula} converge, and let $\xi_{\vec \imath}$ be the vector given by \eqref{eq: master vector formula}.  Our first goal is to show that \eqref{eq: elementary action} holds.
					
					For $w \in \mathcal{W}$, let $\xi_{\vec \imath,w}$ be given by \eqref{eq: master vector formula inductive}, which is the component of $\xi_{\vec \imath}$ in $H_w^\circ$.  Suppose that $v \neq w(1)$, so that $vw$ is also reduced, or equivalent to a word in $\mathcal{W}$.  The reverse computations of \eqref{eq: horrid vector computation} show that \eqref{eq: sum prepared for ind hyp} holds, and hence also \eqref{eq: individual term preparation}.  Letting
					\[
					\eta_{\vec \imath,b,w,v} = \frac{1}{t_b^{(v)}} \xi_{\operatorname{swap}_v^b(\vec \imath),w},
					\]
					we see that \eqref{eq: some sums equal} holds.  Then since $e_{i(v),b}^{(v)} = t_{i(v)}^v \mathbf{1}_{i(v)=b} t_{i(v)}^v + \mathring{e}_{i(v),b}^{(v)}$, we deduce \eqref{eq: xi decomposition}, that is,
					\[
					U_v(\xi_{\vec \imath,w} \oplus \xi_{\vec \imath,vw}) = \sum_{b' \in [n(v)]} e_{i(v),b'}^{(v)} \otimes \frac{1}{t_{b'}^{(v)}} \xi_{\operatorname{swap}_v^{b'}(\vec \imath),w}.
					\]
					Therefore,
					\begin{align*}
						(e_{a,b}^{(v)} \otimes 1) U_v(\xi_{\vec \imath,w} \oplus \xi_{\vec \imath,vw}) &= \mathbf{1}_{b=i(v)} \sum_{b' \in [n(v)]} e_{a,b'}^{(v)} \otimes \frac{1}{t_{b'}^{(v)}} \xi_{\operatorname{swap}_v^{b'}(\vec \imath),w} \\
						&= \mathbf{1}_{b=i(v)} U_v(\xi_{ \operatorname{swap}_v^a(\vec \imath),w} \oplus \xi_{ \operatorname{swap}_v^a(\vec \imath),vw}).
					\end{align*}
					Now taking the direct sum of the this relation over $w \in \mathcal{W}$, we see that
					\begin{equation} \label{eq: elementary action converse}
						e_{a,b}^{(v)} \xi_{\vec \imath} = \mathbf{1}_{b=i(v)} \xi_{\operatorname{swap}_v^a(\vec \imath)},
					\end{equation}
					that is, \eqref{eq: elementary action} holds.  In particular, this means that $e_{i(v),i(v)}^{(v)} \xi_{\vec \imath} = \xi_{\vec \imath}$ for each $v \in \cV$, and so
					\begin{equation} \label{eq: range relation converse}
						\xi_{\vec \imath} \in \Ran(p_{\vec \imath}).
					\end{equation}
					By applying \eqref{eq: elementary action converse} iteratively for each $v \in \cK$, we also see that
					\[
					\xi_{\vec \imath} = e_{\vec \imath, \vec \jmath}^{(\cK)} \xi_{\vec \jmath}
					\]
					for all $\vec \imath, \vec \jmath \in [\vec n]$.  Combining this with \eqref{eq: range relation converse}, we get
					\[
					\xi_{\vec \imath} = e_{\vec \imath,\vec \jmath}^{(\cK)} \xi_{\vec \jmath} \in \Ran(e_{\vec \imath,\vec \jmath}^{(\cK)} p_{\vec \jmath} e_{\vec \jmath,\vec \imath}^{(\cK)}),
					\]
					and since this holds for all $\vec \jmath$, we finally obtain $\xi_{\vec \imath} \in \Ran(q_{\vec \imath})$ by definition of $q_{\vec \imath}$ (Notation \ref{nota: fd summands in gp 2}).  Thus, \eqref{eq: master vector setup} holds as desired.
				\end{proof}
				
				\begin{lemma} \label{lem: Hilbert space computation of projections}
					For $v \in \cV$, let $s^{(v)} = \sum_{b \in [n(v)]} 1 / t_b^{(v)}$.  Then
					\begin{equation} \label{eq: summation condition}
						q_{\vec{\imath}} \neq 0 \iff \sum_{\ell \in \bN_0} \sum_{w \in \mathcal{W}_\ell} (s^{(w_1)} - 1) \dots (s^{(w_\ell)} - 1) < \infty.
					\end{equation}
					In the case that $q_{\vec{\imath}} \neq 0$, we have
					\begin{equation} \label{eq: state on projection formula}
						\frac{1}{\varphi(q_{\vec{\imath}})} = \sum_{\ell \in \bN_0} \sum_{w \in \mathcal{W}_\ell}(s^{(w_1)}-1) \dots (s^{(w_\ell)} - 1) \prod_{v \in \cV} \frac{1}{t_{i(v)}^{(v)} s^{(v)}}.
					\end{equation}
				\end{lemma}
				
				\begin{proof}
					We will start with a vector $\xi_{\vec{1}} \in \Ran(q_{\vec1})$ and compute the norm of each $\xi_{\vec\imath} = e_{\vec\imath,\vec1}^{(\cK)}\xi_{\vec1}$ using the orthogonal decomposition \eqref{eq: master vector formula}, then show that it can only be non-zero and finite if the sum above converges.
					
					Note that the terms $\ip{\delta_{\vec{\imath}}, \varepsilon_{a_1,b_1}^{(w_1)} \dots \varepsilon_{a_\ell,b_\ell}^{(w_\ell)} \delta_{\vec{\jmath}}}$ are zero or one, so they are unchanged under squaring.
					We also have that
					\[
					\text{for } a \neq b, \quad \norm{ (t_b^{(v)})^{-1} \mathring{e}_{a,b}^{(v)} }_{\varphi}^2 = (t_b^{(v)})^{-2} \varphi((e_{a,b}^{(v)})^* e_{a,b}^{(v)}) = (t_b^{(v)})^{-2} \varphi(e_{b,b}^{(v)}) = (t_b^{(v)})^{-1}
					\]
					and
					\[
					\norm{ (t_b^{(v)})^{-1} \mathring{e}_{b,b}^{(v)} }_{\varphi}^2 = (t_b^{(v)})^{-2} \norm{e_{b,b}^{(v)} - \varphi(e_{b,b}^{(v)}) 1}_\varphi^2 = (t_b^{(v)})^{-2} t_b^{(v)}(1 - t_b^{(v)}) = (t_b^{(v)})^{-1} - 1.
					\]
					We therefore have
					\begin{equation} \label{eq: master norm formula}
						\norm{\xi_{\vec{\imath}}}^2 = \sum_{\vec{\jmath} \in [\vec n]} |\alpha_{\vec{\jmath}}|^2 \sum_{\ell \in \bN_0} \sum_{w \in \mathcal{W}_\ell} \sum_{\substack{a_1,b_1 \in [n(w_1)] \\ \vdots \\ a_{\ell}, b_{\ell} \in [n(w_\ell)]}} \ip{\delta_{\vec{\imath}}, \varepsilon_{a_1,b_1}^{(w_1)} \dots \varepsilon_{a_\ell,b_\ell}^{(w_\ell)} \delta_{\vec{\jmath}}} \, \prod_{t=1}^\ell \left( \frac{1}{t_{b_t}^{(w_t)}} - \mathbf{1}_{a_t=b_t} \right).
					\end{equation}
					Define for $v \in \cV$ the matrix
					\[
					S^{(v)} = \sum_{a,b \in [n(v)]} \frac{1}{t_{b}^{(v)}} \varepsilon_{a,b}^{(v)} \in \bigotimes_{v \in \cV} \mathbb{M}_{n(v)},
					\]   
					and note that
					\[
					\sum_{a,b \in [n(v)]} \varepsilon_{a,b}^{(v)}\left( \frac{1}{t_{b}^{(v)}} - \mathbf{1}_{a=b} \right) = S^{(v)} - 1.
					\]
					Then
					\[
					\sum_{\substack{a_1,b_1 \in [n(w_1)] \\ \vdots \\ a_{\ell}, b_{\ell} \in [n(w_\ell)]}} \ip{\delta_{\vec{\imath}}, \varepsilon_{a_1,b_1}^{(w_1)} \dots \varepsilon_{a_\ell,b_\ell}^{(w_\ell)} \delta_{\vec{\jmath}}} \, \prod_{t=1}^\ell \left( \frac{1}{t_{b_t}^{(w_t)}} - \mathbf{1}_{a_t=b_t} \right) = \ip{\delta_{\vec{\imath}}, (S^{(w_1)}-1) \dots (S^{(w_\ell)} - 1) \delta_{\vec{\jmath}}},
					\]
					and so
					\begin{equation} \label{eq: master norm formula 2}
						\norm{\xi_{\vec{\imath}}}^2 = \sum_{\ell \in \bN_0} \sum_{w \in \mathcal{W}_\ell} \ip*{\delta_{\vec{\imath}}, (S^{(w_1)}-1) \dots (S^{(w_\ell)} - 1) \sum_{\vec{\jmath}} |\alpha_{\vec{\jmath}}|^2 \delta_{\vec{\jmath}}},
					\end{equation}
					where all the terms in the summation are nonnegative.
					We further note the following properties of the matrices $S^{(v)}$, recalling that we set $s^{(v)} = \sum_{b\in[n(v)]}\frac1{t_b^{(v)}}$:
					\begin{itemize}
						\item $S^{(v)} - 1$ has nonnegative entries.
						\item $S^{(v)}$ is a rank one matrix in $\mathbb{M}_{n(v)}$ tensor the identity in the other factors.
						\item $(S^{(v)})^2 = s^{(v)} S^{(v)}$.
						\item $S^{(v)}$ and $S^{(v')}$ commute for all $v \neq v'$, and in fact they are in tensor position.
						\item $\sum_{\vec{\imath}} \delta_{\vec{\imath}}$ is an eigenvector of each $S^{(v)}$ with eigenvalue $s^{(v)}$.
						\item Moreover, the vector
						\[
						\gamma = \sum_{\vec{\imath} \in [\vec n]} \prod_{v \in \cV} \frac{1}{t_{i(v)}^{(v)}} \delta_{\vec{\imath}}
						= \bigotimes_{v \in \cV} \sum_{i \in [n(v)]} \frac1{t_{i(v)}^{(v)}} \delta_{i(v)}
						\]
						is an eigenvector of $(S^{(v)})^*$ with eigenvalue $s^{(v)}$.
					\end{itemize}
					
					Recall that $\xi_{\vec{\imath}}$ was obtained by applying the partial isometry $e_{\vec{\imath},\vec{1}}^{(\cK)}$ to the vector $\xi_{\vec{1}}$ which is in the range of $q_{\vec{\imath}}$, and hence in the source subspace for this partial isometry.  Hence, $\norm{\xi_{\vec{\imath}}}$ is the same for all $\vec{\imath}$.  We will thus simplify our computation by taking linear combinations of the norms to exploit the left eigenvector $\gamma$ above.
					Namely, for each $\vec{k}$, using \eqref{eq: master norm formula 2} we have
					\begin{align*}
						\prod_{v \in \cV} s^{(v)} \norm{\xi_{\vec{k}}}^2 &= \sum_{\vec{\imath} \in [\vec n]} \left( \prod_{v \in \cV} \frac{1}{t_{i(v)}^{(v)}} \right) \norm{\xi_{\vec{\imath}}}^2 \\
						&= \sum_{\ell \in \bN_0} \sum_{w \in \mathcal{W}_\ell} \ip*{\gamma, (S^{(w_1)}-1) \dots (S^{(w_\ell)} - 1) \sum_{\vec{\jmath}} |\alpha_{\vec{\jmath}}|^2 \delta_{\vec{\jmath}}} \\
						&= \left( \sum_{\ell \in \bN_0} \sum_{w \in \mathcal{W}_\ell} (s^{(w_1)}-1) \dots (s^{(w_\ell)} - 1) \right) \ip*{\gamma, \sum_{\vec{\jmath}} |\alpha_{\vec{\jmath}}|^2 \delta_{\vec{\jmath}}} \\
						&= \left( \sum_{\ell \in \bN_0} \sum_{w \in \mathcal{W}_\ell} (s^{(w_1)}-1) \dots (s^{(w_\ell)} - 1) \right) \sum_{\vec\jmath \in [\vec n]} \prod_{v \in \cV} \frac{1}{t_{j(v)}^{(v)}} |\alpha_{\vec{\jmath}}|^2.
					\end{align*}
					In particular, if the summation in \eqref{eq: summation condition} is infinite, then the only way this equation can hold is if $\alpha_{\vec{\jmath}} = 0$ for all $\vec{\jmath}$ and so $\xi_{\vec{\imath}} = 0$ for all $\vec{\imath}$; thus, in this case $q_{\vec{\imath}} = 0$ for all $\vec{\imath}$.  On the other hand, if the summation converges, then \eqref{eq: master vector formula} produces vectors in the range of $q_{\vec{\imath}}$ by the converse claim in Lemma \ref{lem: master vector lemma}.  Thus, $q_{\vec{\imath}}$ is nonzero for every $\vec{\imath}$.  This proves \eqref{eq: summation condition}.
					
					To prove \eqref{eq: state on projection formula}, fix $\vec{\imath}$ and apply the above computations explicitly taking $\xi_{\vec{\imath}} = q_{\vec \imath}\Omega$ (and so $\xi_{\vec{\jmath}} = e_{\vec{\jmath},\vec{\imath}}^{(\cK)} \xi_{\vec{\imath}}$ for $\vec{\jmath} \neq \vec{\imath}$).  In the context of the decomposition \eqref{eq: master vector formula} from Lemma~\ref{lem: master vector lemma}, we have that
					\[
					\alpha_{\vec{\jmath}} = \ip{\Omega, \xi_{\vec{\jmath}}} = \varphi(e_{\vec{\jmath},\vec{\imath}}^{(\cK)} q_{\vec{\imath}}).
					\]
					Let $\sigma_t^\varphi$ be the modular group.  Recall $\sigma_t^\varphi|_{M_v} = \sigma_t^{\varphi_v}$. Since we chose $\varphi_v |_{\mathbb{M}_{n(v)}}$ to be given by a diagonal matrix, we have that $e_{j(v),i(v)}^{(v)}$ is an eigenvector of the modular group.
					Hence, so is the product $e_{\vec{\jmath},\vec{\imath}}^{(\cK)}$ and so for some $\lambda$ we have
					\[
					\varphi(e_{\vec{\jmath},\vec{\imath}}^{(\cK)} x) = \lambda \varphi(x e_{\vec{\jmath},\vec{\imath}}^{(\cK)}).
					\]
					In particular,
					\[
					\alpha_{\vec{\jmath}} 
					= \varphi(e_{\vec{\jmath},\vec{\imath}} q_{\vec{\imath}}) = \lambda \varphi(q_{\vec{\imath}} e_{\vec{\jmath},\vec{\imath}}) = 0 \text{ for } \vec{\jmath} \neq \vec{\imath}.
					\]
					Meanwhile,
					\[
					\alpha_{\vec{\imath}} = \varphi(q_{\vec{\imath}}).
					\]
					This results in
					\[
					\prod_{v \in \cV} s^{(v)} \varphi(q_{\vec{\imath}}) = \prod_{v \in \cV} s^{(v)} \norm{q_{\vec{\imath}}}_\varphi^2= \sum_{\ell \in \bN_0} \sum_{w \in \mathcal{W}_\ell}(s^{(w_1)}-1) \dots (s^{(w_\ell)} - 1) \prod_{v \in \cV} \frac{1}{t_{i(v)}^{(v)}} \varphi(q_{\vec{\imath}})^2
					\]
					and therefore,
					\[
					\frac{1}{\varphi(q_{\vec{\imath}})} = \sum_{\ell \in \bN_0} \sum_{w \in \mathcal{W}_\ell}(s^{(w_1)}-1) \dots (s^{(w_\ell)} - 1) \prod_{v \in \cV} \frac{1}{t_{i(v)}^{(v)} s^{(v)}},
					\]
					which proves \eqref{eq: state on projection formula}.
				\end{proof}

				\subsection{Realization of all atomic summands} \label{subsec: realization of all}
				
				Our next goal is to show that all type I factor summands arise from the construction in the previous sections.
				
				\begin{notation} \label{nota: general f d summand}
					Let $(M,\varphi) = \gp_{v \in \cG} (M_v,\varphi_v)$.  In the rest of this section, assume that $(N,\psi)$ is a direct summand with $N \cong B(H)$ for some Hilbert space $H$.  Let $\pi: M \to N$ be the projection.  For each $v \in \cV$, note that $\ker(\pi \circ \iota_v)$ is an ideal in $M_v$, hence a direct summand.  Let $N_v$ be the complementary direct summand in $M_v$, and let $\pi_v: M_v \to N_v$ be the projection.  Let $1_N$ and $1_{N_v}$ denote the internal units of $N$ and $N_v$ respectively, which are central projections in $M$ and $M_v$ respectively.
				\end{notation}
				
				\begin{lemma}
					With the setup of Notation \ref{nota: general f d summand}, $\pi|_{N_v}$ is unital and injective and $N_v$ is atomic, i.e.\ a direct sum of type I factors.
				\end{lemma}
				
				\begin{proof}
					By construction, $M_v \cong N_v \oplus \ker(\pi|_{M_v})$ and so $N_v \cap \ker(\pi) = 0$, or $\pi|_{N_v}$ is injective.  Also, $1 - 1_{N_v} \in \ker(\pi|_{M_v})$, so $\pi|_{N_v}$ is unital.
					
					To show that $N_v$ is atomic, we want to construct a conditional expectation from $N$ onto $N_v$.  By \cite[Remark 2.24]{CaFi2017}, there is a faithful normal state-preserving conditional expectation $E_{M_v}: M \to M_v$.  Note that $y = E_{M_v}[1_N] \in Z(M_v)$ because $1_N \in Z(M)$. Observe that for $x \in N_v$,
					\[
					E_{M_v}[\pi(x)] = E_{M_v}[x 1_N] = x E_{M_v}[1_N] = xy.
					\]
					In particular, $(1 - 1_{N_v}) y = 0$, so $y \leq 1_{N_v}$; equivalently, $y \in N_v$ because $N_v$ is an ideal in $M_v$.  On the other hand, letting $p$ be the kernel projection of $y$ as an element of $N_v$, then $E_{M_v}[\pi(p)] = py = 0$; by faithfulness of $E_{M_v}$ and injectivity of $\pi|_{N_v}$, we have $p = 0$.  Hence, the support projection of $y$ is equal to $1_{N_v}$.
					
					We claim that there is a normal conditional expectation $\varepsilon: N \to N_v$ such that $E_{M_v}[x] = y \varepsilon[x]$ for $x \in N_v$.  For $k \in \bN$, let $p_k= \mathbbm{1}_{(1/(k+1),1/k]}(y)$ be the spectral projection of $y$ associated to $(1/(k+1),1/k]$.  Since $y \leq 1$, we have $\sum_{k \in \bN} p_k = 1_{N_v}$.  Moreover, there is a unique positive element $z_k \in Z(N_v)$ with $z_k y = p_k$. Let $\varepsilon_k: N \to N_v$ be given by $\varepsilon_k[x] = z_k E_{M_v}(x)$.  This is a unital, normal, completely positive map from $N$ to $p_k N_v$.  Since $N \cong \bigoplus_k p_k N_v$, the map $\varepsilon: N \to N_v$ given by $\varepsilon(x) = \sum_k \varepsilon_k(x)$ is unital, normal, and completely positive.  Moreover, for $x \in M$, we have
					\[
					y \varepsilon[x] = \sum_k yz_k E_{M_v}[\pi(x)] = \sum_k p_k E_{M_v}[\pi(x)] = E_{M_v}[\pi(x)].
					\]
					In particular, for $x \in N_v$, we have $y \varepsilon \circ \pi(x) = E_{M_v}[\pi(x)] = y \pi(x)$, and since the support projection of $y$ is equal to $1_N$, we conclude
					\[
					\pi(x) = 1_N \pi(x) = 1_N \varepsilon \circ \pi(x) = \pi((\varepsilon \circ \pi(x))),
					\]
					and since $\pi|_{N_v}$ is injective, we conclude that $\varepsilon \circ \pi(x) = x$.  Hence, $\varepsilon \circ \pi = \id$ and so $\varepsilon$ gives a normal conditional expectation $N \to N_v$ corresponding to the inclusion $\pi: N_v \to N$.  By \cite[Theorem IV.2.2.2]{Blackadar2006encyclopedia}, since there is a normal conditional expectation from the type I factor $N$ onto $N_v$, we know $N_v$ is atomic.
				\end{proof}
				
				\begin{lemma}
					With the setup of Notation \ref{nota: general f d summand}, suppose that $v$ and $v'$ are not adjacent.  Then either $N_v$ or $N_{v'}$ is $1$-dimensional.
				\end{lemma}
				
				\begin{proof}
					Suppose for contradiction that $N_v$ and $N_{v'}$ are both at least $2$-dimensional.  Then we can write $1_{N_v} = e_1 + e_2$ where $e_1$ and $e_2$ are projections in the centralizer of the state $\psi_v$.  Indeed, if $N_v$ has nontrivial center, $e_1$ and $e_2$ can be chosen to be central, and otherwise $N_v \cong B(H)$ and the state is given by a density operator, so choose $e_1$ and $e_2$ to be projections that commute with it.  Similarly, write $1_{N_{v'}} = f_1 + f_2$ where $f_1$ and $f_2$ are projections in the centralizer of $\psi_{v'}$.  Also let $e_0 = 1 - 1_{N_v}$ and $f_0 = 1 - 1_{N_{v'}}$.
					
					Let $A_v = \Span(e_0,e_1,e_2)$ and $A_{v'} = \Span(f_0,f_1,f_2)$, and let $A = A_v \vee A_{v'} \subseteq M$. Note that $A_v$ and $A_{v'}$ are freely independent, so by \cite[Theorem 2.3]{Dykema1993freedimension}, if $e_0$ and $f_0$ are not both zero, then
					\[
					A \cong A_v * A_{v'} \cong \gamma L(F_s) \oplus \bigoplus_{i,j} \gamma_{i,j} \C,
					\]
					where $s \in (1,\infty)$, and the weights $\gamma_{i,j}$ are given by $\gamma_{i,j} = \max(0,\varphi(e_i) + \varphi(f_j) - 1)$, and $\gamma = 1 - \sum_{i,j} \gamma_{i,j}$; by convention, when one of the coefficients is zero, the corresponding direct summand vanishes.  Moreover, the $(i,j)$ summand in the decomposition is exactly the span of $e_i \wedge f_j$.  In case $e_0$ and $f_0$ are both zero, we similarly have by \cite[Theorem 1.1]{Dykema1993freedimension} that
					\[
					A \cong A_v * A_{v'} \cong \gamma(\mathbb{M}_2 \otimes L^\infty[0,1]) \oplus \bigoplus_{i,j} \gamma_{i,j} \C,
					\]
					where again $\gamma_{i,j} = \max(0,\varphi(e_i) + \varphi(f_j) - 1)$ and the $(i,j)$ summand is the span of $e_i \wedge f_j$.
					
					We claim that there is a state-preserving conditional expectation from $M$ onto $A$.  Indeed, since $A_v \subseteq M_v^{\varphi_v}$, there is a state-preserving conditional expectation from $M_v$ to $A_v$; similarly, from $M_{v'}$ to $A_{v'}$.  By taking the free product of these expectations by \cite{Boca1993}, there is a normal completely positive map $M_v * M_{v'} \to A_v * A_{v'}$ that restricts to the identity on $A_v * A_{v'}$, and in fact this is a state-preserving conditional expectation.  By composing with the conditional expectation from $M$ to $M_v * M_{v'}$, we obtain a state-preserving normal conditional expectation $E_A: M \to A$.
					
					Similarly to the previous proof, $E_A[\pi(x)] = E_A[1_N] x$ for $x \in A$.  Let $y = E_A[1_N] \in Z(A)$, and let $p$ be the support projection of $y$.  Then $x \in A$ is in $\ker(\pi)$ if and only if $xy = 0$, so that $\ker(\pi) \cap A = (1-p)A$.  We thus see that $\pi$ restricts to an injective $*$-homomorphism $pA \to N$.  Since $N$ is an ideal in $M$ and $\pi$ is an $M$-$M$-bimodule map, we have that $\pi(p) = p \pi(1) = \pi(1) p$, so that $p$ and $\pi(1) = 1_N$ commute, and so $pNp = (p \wedge 1_N) N (p \wedge 1_N)$ is a corner of $N$ and so also a type I factor.  Moreover, $\pi(p) = p \wedge 1_N = 1_{pNp}$, so $\pi$ gives a unital $*$-homomorphism $pA \to pNp$.
					
					Next, by the same argument as in the previous proof, there exists a unital, normal, completely positive map $\varepsilon: N \to pA$ such that $E_A[x] = y \varepsilon[x]$ for $x \in N$.  Thus, for $x \in pA$, we have $yx = E_A[\pi(x)] = y \varepsilon[x]$, which implies that $x = \varepsilon[x]$ since the support projection of $y$ is $p$ and $x \in pA$.  Therefore, $\varepsilon$ provides a normal conditional expectation $pNp \to pA$ corresponding to the inclusion $\pi|_{pA}: pA \to pNp$.  Since $pNp$ is a type I factor, \cite[Theorem IV.2.2.2]{Blackadar2006encyclopedia} shows that $p A$ is atomic.  So the intersection of $pA$ with the diffuse part of $A$ (either $L(F_s)$ or $\mathbb{M}_2 \otimes L^\infty[0,1]$) is trivial, or equivalently $p \leq \sum_{i,j=0}^2 e_i \wedge f_j$.
					
					Now assume that $\max_i \varphi(e_i) \geq \max_j \varphi(f_j)$ by switching $v$ and $v'$ if necessary, and let $i$ be an index where $\varphi(e_i)$ is maximal.  Then for $i' \neq i$ and for $j \in \{0,1,2\}$, we have $\varphi(e_{i'}) + \varphi(f_j) - 1 \leq \varphi(e_{i'}) + \varphi(e_i) - 1 \leq 0$.  Therefore, $\gamma_{i',j} = 0$ whenever $i' \neq i$. Thus, $e_{i'} \leq 1 - p$, so $e_{i'} \in (1-p)A \subseteq \ker(\pi)$.  This implies that $e_{i'} \in \ker(\pi)$ for two distinct indices $i'$ in $\{0,1,2\}$, which contradicts the fact that $e_1$ and $e_2$ are in $N_v$ and hence not in the kernel of $\pi$.
				\end{proof}
				
				\begin{lemma} \label{lem: general scenario clique}
					Assume the setup of Notation \ref{nota: general f d summand}.  Then
					\begin{enumerate}
						\item $\{v \in \cV: \dim N_v > 1\}$ is a clique in $\cG$.
						\item The images $\pi(M_v)$ commute in $N$.
						\item Each $N_v$ is a factor.
					\end{enumerate}
				\end{lemma}
				
				\begin{proof}
					(1) is immediate from the previous lemma.
					
					(2) Let $\cK$ be the clique from (1).  For $v \in \cK$, the algebras $M_v$ commute since the vertices are adjacent.  Since $\pi$ is a $*$-homomorphism, the images $\pi(M_v)$ commute for $v \in \cK$.  For $v \not \in \cK$, we have $\pi(M_v) = \pi(N_v)$ is $1$-dimensional, and so it equals $\C 1_N$, so $\pi(M_v)$ commutes with everything in $N$.
					
					(3) Suppose that $p$ is a central projection in $N_v$.  Then $\pi(p)$ is a central projection in $N$.  We assumed that $N$ is a factor, so $\pi(p) = 1_N$.  Since $\pi|_{N_v}$ is injective, $p = 1$.  Thus, $N_v$ is a factor.
				\end{proof}
				
				\begin{lemma} \label{lem: infinite-dimensional join decomposition}
					Assume the setup of Notation \ref{nota: general f d summand}.  Suppose that $N_v$ is infinite-dimensional and $v'$ is another vertex in $\cG$ with $M_{v'} \neq \C$.  Then $v$ is adjacent to $v'$.
				\end{lemma}
				
				\begin{proof}
					Suppose for contradiction that $v$ is not adjacent to some $v'$.  By the previous lemma, $N_{v'}$ is $1$-dimensional.  Since we assume that $M_{v'} \neq \C$, we have $\varphi(1_{N_{v'}}) < 1$.  Recall $N_v \cong B(H)$ with $H$ infinite-dimensional, and since $N_v$ has a faithful normal state, $H$ must be separable.  So there is a sequence of minimal projections $(p_k)_{k \in \bN}$ that add up to $1$.  We also have $\lim_{k \to \infty} \varphi(p_k) = 0$, so for sufficiently large $k$, $\varphi(p_k) + \varphi(1_{N_{v'}}) < 1$, which implies that $p_k \wedge 1_{N_{v'}} = 0$ since they are freely independent.  Now $\pi(p_k) = p_k 1_N = p_k \pi(p_k)$ so that $\pi(p_k) \leq p_k$.  Moreover, $\pi(p_k) \leq 1_N = \pi(1_{N_{v'}}) \leq 1_{N_{v'}}$.  Thus, $\pi(p_k) \leq p_k \wedge 1_{N_{v'}} = 0$.  So $p_k \in \ker(\pi)$ which contradicts that $\pi$ is injective on $N_v$.
				\end{proof}
				
				In particular, assuming that $M_v \neq \C$ for all $v \in \cV$ and letting $\cK' = \{v \in \cV: \dim N_v  = \infty\}$, we see that $\cK'$ is a clique and $\cG$ splits as a graph join of $\cK'$ with the complementary graph $\cG'$, and consequently,
				\[
				(M,\varphi) = \bigotimes_{v \in \cK'} (M_v,\varphi_v) \otimes \gp_{v \in \cG'} (M_v,\varphi_v).
				\]
				Thus, the classification of atomic summands for $(M,\varphi)$ can be reduced to that of $\gp_{v \in \cG'} (M_v,\varphi_v)$ (see the conclusion of the proof of Theorem \ref{thm: classification of B(H) summands} in \S \ref{subsec: conclusion of proof} for details).  So we now classify the summands for the case where the $N_v$'s are finite-dimensional.
				
				\begin{lemma} \label{lem: realization of all f d summands}
					Let $N$ be a type I factor summand of the graph product as in Notation \ref{nota: general f d summand}, and assume that $N_v$ is finite-dimensional for each $v$.  Then $N$ has the form given in Lemma \ref{lem: construction of f d summand and state}.
				\end{lemma}
				
				\begin{proof}
					By Lemma \ref{lem: general scenario clique}, $\cK = \{v \in \cV: \dim N_v > 1\}$ is a clique, and $N_v \cong \mathbb{M}_{n(v)}$ for some $n(v) \in \bN$.  Thus, assume that $N_v$ is given by matrix units as in Notation \ref{nota: fd summands in gp}, and recall $p_{\vec{\imath}} = \bigwedge_{v \in \cV} e_{i(v),i(v)}^{(v)}$.  Let $\tilde{N}$ be the finite-dimensional summand constructed in Lemma \ref{lem: construction of f d summand and state}.  We want to show that $N \subseteq \tilde{N}$.
					
					Since normal $*$-homomorphisms respect minima of projections, we have
					\[
					\pi(p_{\vec{\imath}}) = \bigwedge_{v \in \cV} \pi(e_{i(v),i(v)}^{(v)}) = \prod_{v \in \cV} \pi(e_{i(v),i(v)}^{(v)}) = \pi(e_{\vec{\imath},\vec{\imath}}^{(\cK)})
					\]
					where $e_{\vec{\imath},\vec{\jmath}}^{(\cK)}$ is as in Notation \ref{nota: fd summands in gp 2}.
					
					Note that for $x, y \in M$, we have $\pi(xy) = \pi(x) y = x \pi(y)$.  In particular, if $p$ is a projection, then $\pi(p) = \pi(p^2) = p \pi(p)$ and therefore $\pi(p) \leq p$.  Hence, $\pi(p_{\vec{\imath}}) \leq p_{\vec{\imath}}$.  We also have
					\[
					\pi(e_{\vec{\imath}, \vec{\jmath}}^{(\cK)} p_{\vec{\jmath}} e_{\vec{\jmath}, \vec{\imath}}^{(\cK)}) = e_{\vec{\imath}, \vec{\jmath}}^{(\cK)} \pi(p_{\vec{\jmath}}) e_{\vec{\jmath}, \vec{\imath}}^{(\cK)} = e_{\vec{\imath}, \vec{\jmath}}^{(\cK)} \pi(e_{\vec{\jmath},\vec{\jmath}}^{(\cK)}) e_{\vec{\jmath}, \vec{\imath}}^{(\cK)} = \pi(e_{\vec{\imath}, \vec{\jmath}}^{(\cK)} e_{\vec{\jmath},\vec{\jmath}}^{(\cK)} e_{\vec{\jmath}, \vec{\imath}}^{(\cK)}) = \pi(e_{\vec{\imath}, \vec{\imath}}^{(\cK)}) = \pi(p_{\vec{\imath}}).
					\]
					Therefore,
					\[
					\pi(p_{\vec{\imath}}) = \pi(e_{\vec{\imath}, \vec{\jmath}}^{(\cK)} p_{\vec{\jmath}} e_{\vec{\jmath}, \vec{\imath}}^{(\cK)}) \leq e_{\vec{\imath}, \vec{\jmath}}^{(\cK)} p_{\vec{\jmath}} e_{\vec{\jmath}, \vec{\imath}}^{(\cK)}.
					\]
					Taking the intersection over $\vec{\jmath} \in [\vec n]$, we obtain
					\[
					\pi(e_{\vec{\imath},\vec{\imath}}^{(\cK)}) = \pi(p_{\vec{\imath}}) \leq \bigwedge_{\vec{\jmath} \in [\vec n]} e_{\vec{\imath}, \vec{\jmath}}^{(\cK)} p_{\vec{\jmath}} e_{\vec{\jmath}, \vec{\imath}}^{(\cK)} =  q_{\vec{\imath}}.
					\]
					Therefore,
					\[
					1_N = \pi\left( \prod_{v \in \cK} 1_{N_v} \right) = \sum_{\vec{\imath} \in [\vec n]} \pi(e_{\vec{\imath},\vec{\imath}}^{(\cK)}) \leq \sum_{\vec{\imath} \in [\vec n]} q_{\vec{\imath}} = 1_{\tilde{N}}.
					\]
					Since $N$ and $\tilde{N}$ are ideals, this implies that $N \subseteq \tilde{N}$.  Since $\tilde{N}$ is a factor, it is a minimal nonzero ideal in $M$, and therefore $N = \tilde{N}$ as desired.
				\end{proof}
				
				\section{Evaluation of the power series} \label{sec: power series}
				
				\subsection{Cartier and Foata's formula} \label{subsec: power series formula}
				
				In Lemma \ref{lem: Hilbert space computation of projections}, we showed that the projections $q_{\vec{\imath}}$ are nonzero if and only if
				\[
				\sum_{\ell \in\bN_0} \sum_{w \in \mathcal{W}_\ell} (s^{(w_1)} - 1) \dots (s^{(w_\ell)} - 1) < \infty.
				\]
				Our next goal is to give criteria for convergence of the power series
				\[
				f(X) = \sum_{\ell \in \bN_0} \sum_{w \in \mathcal{W}_\ell} X_{w_1} \dots X_{w_\ell}, \text{ where } X = (X_v)_{v \in \cV}.
				\]
				We will accomplish this by relating it to the power series summing over all words up to $\cG$-equivalence:
				\[
				\tilde{f}(X) = \sum_{w \in \overline\cW} X_{w_1} \dots X_{w_\ell},
				\]
				where $\overline\cW$ is a set of representatives for the $\cG$-equivalence classes of words.
				The relationship between these two as formal power series follows by a fairly standard combinatorial argument; see, e.g., \cite[esp. \S2.2]{GJ2004}.
				Notice that an arbitrary word can be expressed uniquely by beginning with a reduced word and replacing each letter by one or more copies of that same letter.
				Since $t_v/(1 - t_v) = \sum_{k \in \bN} t_v^k$, this decomposition gives us the identity
				\[
				\tilde{f}((X_v)_{v \in \cV}) = f\left( \left( \frac{X_v}{1 - X_v} \right)_{v \in \cV} \right).
				\]
				Furthermore, since $z / (1 + z)$ is the inverse function of $z / (1 - z)$ as a formal power series, the substitution can be reversed into
				\[
				f((X_v)_{v \in \cV}) = \tilde{f}\left( \left( \frac{X_v}{1 + X_v} \right)_{v \in \cV} \right).
				\]
				Cartier and Foata \cite[Eqn. (1)]{CF1969} give the very nice equation:
				\begin{equation} \label{eq: CF summation formula}
					\tilde{f}((X_v)_{v \in \cV}) = \left( \sum_{\cK \subseteq \cG \text{clique}} (-1)^{|\cK|} \prod_{v \in \cK} X_v \right)^{-1}.
				\end{equation}
				In fact, this relation even holds for non-commutative power series: if the indeterminates $X_v$ commute according to the graph $\cG$, then
				\[
				\fK_{\cG}(X) = \sum_{\cK \subseteq \cG \text{clique}} (-1)^{|\cK|} \prod_{v \in \cK} X_v
				\]
				is well defined since each term is a product of mutually commuting variables.
				Moreover, $\fK_{\cG}(X)^{-1}$ is the summation of words in the monoid generated by the $X_v$'s, which are identified up to $\cG$-equivalence.
				For example, if $\cG$ has vertex set $[n]$ and all variables commute or no variables commute, they recover the known identities
				\begin{align*}
					\paren*{(1-X_1)\cdots(1-X_n)}^{-1} &= \sum_{\alpha_1, \ldots, \alpha_n} X_1^{\alpha_1}\cdots X_n^{\alpha_n} \\
					\paren*{1 - X_1 - \cdots - X_n}^{-1} &= \sum_{r\in\bN_0}\sum_{i_1, \ldots, i_r} X_{i_1}\cdots X_{i_r}.
				\end{align*}
				
				\subsection{Conclusion of the proof of the main results} \label{subsec: conclusion of proof}
				
				The formal power series relation \eqref{eq: CF summation formula} enables us to give several equivalent criteria for convergence of the series $f(x_v/(1-x_v))$ for real inputs $x_v$ (usually in $[0,1]$).
				
				\begin{lemma} \label{lem: convergence criteria}
					Fix $\cG$, and let $f$, $\tilde{f}$, and $\fK_\cG$ be as above.  Let $(x_v)_{v \in \cV} \in [0,1]^{\cV}$.  Then the following are equivalent:
					\begin{enumerate}[(1)]
						\item $f((x_v/(1-x_v))_{v \in \cV}) < \infty$ (and in particular each $x_v < 1$);
						\item $\tilde{f}(x) < \infty$;
						\item $\fK_{\cG}(y) > 0$ for each $y \in \prod_{v \in \cV} \{0,x_v\}$;
						\item $\fK_{\cG}(y) > 0$ for each $y \in \prod_{v \in \cV} [0,x_v]$;
						\item $\fK_{\cG}(\alpha x) > 0$ for $\alpha \in [0,1]$.
					\end{enumerate}
					Moreover, (2) - (5) are equivalent whenever $x \in [0,\infty)^{\cV}$.  Furthermore, whenever $x \in [0,\infty)^{\cV}$ and $\tilde{f}(x) < \infty$, we have $\tilde{f}(x) = 1/\fK_{\mathcal{G}}(x)$.
				\end{lemma}
				
				\begin{proof}
					(1) $\iff$ (2).
					As mentioned above, viewing $f$ and $\tilde f$ as formal power series we have
					\[
					f((X_v/(1-X_v))_{v \in \cV}) = \tilde{f}(X).
					\]
					If we expand either sum, all the coefficients are non-negative, as are all the inputs $x_v$,
					and hence $f((x_v/(1-x_v))_{v\in \cV}) < \infty$ if and only if $\tilde{f}(x) < \infty$.
					
					For the rest of the proof, we assume that $x \in [0,\infty)^{\cV}$ rather than $[0,1]^{\cV}$.
					
					(2) $\implies$ (3).  Suppose $\tilde{f}(x) < \infty$. By direct comparison the series for $\tilde{f}$ converges absolutely at a complex tuple $z = (z_v)_{v \in \cV}$ if $|z_v| \leq x_v$.  Since $1 / \fK_{\cG}$ agrees with $\tilde{f}$ as a formal power series, we have $\fK_{\cG}(z) f(z) = 1$ at every point $z$ for which two power series converge absolutely, and of course $\fK_{\cG}$ always converges absolutely since it is a polynomial.  In particular, if $y_v \in \{0,x_v\}$, then $\tilde{f}(y)$ converges absolutely and $\fK_{\cG}(y) \tilde{f}(y) = 1$.  Since $\tilde{f}(y)$ has nonnegative terms, we have $\tilde{f}(y) \geq 0$, and therefore $\fK_{\cG}(y) > 0$.
					
					(3) $\implies$ (4).  Note that the terms in $\fK_\cG$ do not repeat the same variable twice, and so $\fK_\cG$ is an affine function of each variable individually.  Therefore, if $y$ and $y' \in [0,\infty)^{\cV}$ agree on $|\cV| - 1$ coordinates, and $\fK_{\cG}(y) > 0$ and $\fK_{\cG}(y') > 0$, then $\fK_{\cG}((1-\lambda)y + \lambda y') > 0$ for $\lambda \in [0,1]$.  Therefore, if we know that $\fK_\cG > 0$ on the corners of the rectangle $\prod_{v \in \cV} [0,t_v]$, then $\fK_\cG > 0$ on the entire rectangle.
					
					(4) $\implies$ (5).  The line segment from $0$ to $x$ is contained in the rectangle $\prod_{v \in \cV} [0,x_v]$.
					
					(5) $\implies$ (2).  We proceed by contrapositive.  Fix $x \in [0,\infty)^{\cV}$ such that $h(x) = +\infty$.  Consider the single-variable power series in $\alpha$ given by
					\[
					h(\alpha) = \tilde{f}(\alpha x) = \sum_{\ell=0}^\infty \left( \sum_{w \in \overline{W}_\ell} x_{w(1)} \dots x_{w(\ell)} \right) \alpha^\ell,
					\]
					where $\overline{W}_\ell$ is the set of words of length $\ell$ in $\overline{\mathcal{W}}$, and the single-variable polynomial
					\[
					p(\alpha) = \fK_{\cG}(\alpha x),
					\]
					so that $h = 1/p$ as formal power series.  Let $r$ be the radius of convergence of $h$.  Since $h(1) = +\infty$ by assumption, we have $r \leq 1$.  By complex analysis, $r$ is the minimal distance of any root of $p$ from $0$.  Thus, there is some $\lambda$ with $|\lambda| = 1$ and $p(r \lambda) = 0$.  Since the power series $h$ converges on the open disk of radius $r$ and since it has nonnegative terms, we have
					\[
					\liminf_{\alpha \to r^+} h(\alpha) \geq \liminf_{\alpha \to r^+} |h(\alpha \lambda)| = \liminf_{\alpha \to r^+} \frac{1}{|p(\alpha \lambda)|} = +\infty.
					\]
					Moreover, by applying the monotone convergence theorem to the series for $h$,
					\[
					h(r) = \lim_{\alpha \to r^+} h(\alpha) = +\infty.
					\]
					Hence, we have
					\[
					p(r) = \lim_{\alpha \to r^+} p(\alpha) = \lim_{\alpha \to r^+} \frac{1}{h(\alpha)} = 0.
					\]
					Therefore, $r \in [0,1]$ and $\fK_{\cG}(rx) = 0$ as desired.
					
					The final claim that $\tilde{f}(x) = 1 / \fK_{\cG}(x)$ whenever the left-hand side was already noted in the proof of (2) $\implies$ (3) above.
				\end{proof}
				
				\begin{proof}[Proof of Theorem \ref{thm: classification of finite-dimensional summands}]
					Consider the same setup as Notations \ref{nota: fd summands in gp} and \ref{nota: fd summands in gp 2}.  In Lemma \ref{lem: Hilbert space computation of projections}, we showed that the projections $q_{\vec{\imath}}$ are nonzero if and only if
					\[
					f((s^{(v)} - 1)_{v \in \cV}) < \infty, \text{ where } s^{(v)} = \sum_{j=1}^{n(v)} \frac{1}{t_j^{(v)}} = \sum_{j=1}^{n(v)} \frac{1}{\alpha^{(v)} \lambda_j^{(v)}}.
					\]
					Let
					\[
					x_v = \frac{s^{(v)} - 1}{(s^{(v)} - 1) + 1} = \frac{s^{(v)} - 1}{s^{(v)}} = 1 - 1/s^{(v)},
					\]
					so that
					\[
					s^{(v)} - 1 = \frac{x_v}{1-x_v}.
					\]
					Therefore, by Lemma \ref{lem: convergence criteria},
					\[
					f((s^{(v)} - 1)_{v \in \cV}) = f\left( \left(\frac{x_v}{1-x_v} \right)_{v \in \cV} \right) < \infty \iff \fK_{\cG}(y) > 0 \text{ for each } y \in \prod_{v \in \cV} \{0,x_v\}.
					\]
					Recall in Theorem \ref{thm: classification of finite-dimensional summands}, we defined
					\[
					\fK_{\cG'}(\vec{x}) = \sum_{\substack{\cK \subseteq \cG' \\ \text{clique}}} (-1)^{|\cK|} \prod_{v \in \cK} x_v.
					\]
					This is exactly the expression obtained by substituting $x_v$ for the variables with $v \in \cV'$ and $0$ for the variables with $v \not \in \cV'$.
					Hence, $q_{\vec{\imath}} \neq 0$ if and only if $\fK_{\cG'}((1 - 1/s^{(v)})_{v \in \cV'}) > 0$ for all induced subgraphs $\cG' \subseteq \cG$.
					
					By Lemma \ref{lem: Hilbert space computation of projections} again (as well as the final claim of Lemma \ref{lem: convergence criteria}), when $q_{\vec{\imath}}$ is nonzero, we have
					\begin{align*}
						\varphi(q_{\vec{\imath}}) &= \prod_{v \in \cV} t_{i(v)}^{(v)} \prod_{v \in \cV} s^{(v)}  \left( \sum_{\ell \in\bN_0} \sum_{w \in \mathcal{W}_\ell}(s^{(w_1)}-1) \dots (s^{(w_\ell)} - 1) \right)^{-1} \\
						&= \prod_{v \in \cV} t_{i(v)}^{(v)} \prod_{v \in \cV} s^{(v)} \fK_{\cG}((1 - 1/s^{(v)})_{v \in \cV}) \\
						&= \prod_{v \in \cV} t_{i(v)}^{(v)} \prod_{v \in \cV} s^{(v)} \sum_{\substack{\cK \subseteq \cG \\ \text{clique}}} (-1)^{|\cK|}\prod_{v\in \cK}(1 - 1/s^{(v)}) \\
						&= \prod_{v \in \cV} t_{i(v)}^{(v)} \sum_{\substack{\cK \subseteq \cG \\ \text{clique}}} \prod_{v \in \cV  \setminus \cK} s^{(v)} \prod_{v \in \cK} (1 - s^{(v)}).
					\end{align*}
					Now recall that the projections $q_{\vec{\imath}}$ sum up to the unit $1_N$ in the direct summand $(N,\psi)$.  Moreover,
					\[
					\sum_{\vec{\imath} \in [\vec n]} \prod_{v \in \cV} t_{i(v)}^{(v)} = \sum_{\vec{\imath} \in [\vec n]} \prod_{v \in \cV} \alpha^{(v)} \lambda_{i(v)}^{(v)} = \prod_{v \in \cV} \left[ \alpha^{(v)} \sum_{i \in [n(v)]} \lambda_i^{(v)} \right] = \prod_{v \in \cV } \alpha^{(v)}.
					\]
					Hence, by summing up our formula for $\varphi(q_{\vec{\imath}})$, we obtain
					\[
					\alpha = \varphi(1_N) = \sum_{\vec{\imath} \in [\vec n]} \varphi(q_{\vec{\imath}}) = \prod_{v \in \cV} \alpha^{(v)} \sum_{\substack{\cK \subseteq \cV \\ \text{clique}}} \prod_{v \in \cV \setminus \cK} s^{(v)} \prod_{v \in \cK} (1 - s^{(v)}),
					\]
					which is the formula asserted in Theorem \ref{thm: classification of finite-dimensional summands}.  Finally, Lemma \ref{lem: realization of all f d summands} shows that all the finite-dimensional factor summands of the graph product must arise from the construction of Lemma \ref{lem: construction of f d summand and state}, as asserted in Theorem \ref{thm: classification of finite-dimensional summands}.
				\end{proof}
				
				\begin{proof}[Proof of Theorem \ref{thm: classification of B(H) summands}]
					Let $(N,\psi)$ be a direct summand in the graph product which is isomorphic to $B(H)$ for $H$ infinite-dimensional.  Using Notation \ref{nota: general f d summand}, let $\cV_1$ be the set of vertices where $(N_v,\psi_v)$ is infinite-dimensional. By Lemma \ref{lem: infinite-dimensional join decomposition}, the vertices in $\cV_1$ are adjacent to all vertices in $\cG$, and so $\cG_1$ is complete and $\cG$ decomposes as a graph join $\cG_1 + \cG_2$.  Consequently,
					\[
					(M,\varphi) = \bigotimes_{v \in \cV_1} (M_v,\varphi_v) \otimes \gp_{v \in \cG_2} (M_v,\varphi_v).
					\]
					By distributing tensor products over direct sums, it is easy to see that $(N,\varphi)$ must be the tensor product of type I factor direct summands $(N_v,\varphi_v)$ in each $(M_v,\varphi_v)$ for $v \in \cV_1$ and a type I factor direct summand $(Q,\omega)$ in the graph product $(M_{\cG_2},\varphi_{\cG_2}) := \gp_{v \in \cG_2} (M_v,\varphi_v)$.  Let $\pi'$ be the projection map $M_{\cG_2} \to Q$ and for $v \in \cV_2$, let $Q_v$ be the complementary summand to $\ker(\pi'|_{M_v})$ as in Notation \ref{nota: general f d summand}. Given the tensor product form of $(N,\varphi)$, we have $\ker(\pi|_{M_v}) = \ker(\pi'|_{M_v})$ for $v \in \cV_2$, and hence $Q_v = N_v$.  In particular, $Q_v$ is finite-dimensional.  Hence, we can apply Lemma \ref{lem: realization of all f d summands} to the direct summand $(Q,\omega)$ in the graph product $(M_{\cG_2},\varphi_{\cG_2})$, so $(Q,\omega)$ is a finite-dimensional factor summand of the form classified by Theorem \ref{thm: classification of finite-dimensional summands}.
				\end{proof}
				
				\begin{proof}[Proof of Theorem \ref{thm: intersection of projections}]
					Let $(M,\varphi)$ be the graph product of $(M_v,\varphi_v)$ for $v \in \cV$.  Let $p_v$ be a projection in $M_v$.  Recall that $\bigwedge_{v \in \cV} p_v \in \mathrm{W}^*((p_v)_{v \in \cV})$, and so whether $\bigwedge_{v \in \cV}p_v$ is nonzero only depends on the behavior of $\mathrm{W}^*((p_v)_{v \in \cV})$ and the state restricted to this von Neumann algebra.  Therefore, it suffices to prove Theorem \ref{thm: intersection of projections} in the case where $M_v = \C p_v \oplus \C(1 - p_v)$.  We are thus in the situation of Notations \ref{nota: fd summands in gp} and \ref{nota: fd summands in gp 2} with $N_v = \C p_v$ and $n(v) = 1$ for all $v$.  Hence, in this case, $\alpha^{(v)} = \varphi_v(p_v)$ and $\lambda_1^{(v)} = 1$ and so $t_1^{(v)} = \alpha^{(v)}$ and $s^{(v)} = 1 / \varphi_v(p_v)$.  Observe that in the notation of Theorem \ref{thm: classification of finite-dimensional summands}, each $N_v$ only has a single associated matrix unit $e_{1,1}^{(v)}$, we have $p_{\vec 1} = \bigwedge_{v \in \cV} p_v$, the clique $\cK$ is empty, and so
					\[
					1_N = q_{\vec 1} = p_{\vec 1} = \bigwedge_{v \in \cV} p_v.
					\]
					Hence, by Lemma \ref{lem: Hilbert space computation of projections} and Lemma \ref{lem: convergence criteria}, as in the proof of Theorem \ref{thm: classification of finite-dimensional summands}, we have
					\[
					\bigwedge_{v \in \cV } p_v \neq 0 \iff \forall \cG' \subseteq \cG, \fK_{\cG'}((1 - 1/s^{(v)})_{v \in \cV'}) = \fK_{\cG'}((1 - \varphi_v(p_v))_{v \in \cV'}) > 0,
					\]
					where $\cG'$ ranges over induced subgraphs of $\cG$ with vertex set $\cV'$.
					Moreover, the last formula in Theorem \ref{thm: classification of finite-dimensional summands} reduces to
					\[
					\alpha = \varphi(1_N) = \varphi \left( \bigwedge_{v \in \cV } p_v \right) = \fK_{\cG}((1 - \varphi_v(p_v))_{v \in \cV})
					\]
					as asserted in Theorem \ref{thm: intersection of projections}.
				\end{proof}
				
				\subsection{Geometry of the region of convergence} \label{subsec: region}
				
				In light of Theorem \ref{thm: intersection of projections} and Theorem \ref{thm: classification of finite-dimensional summands}, to determine whether a graph product has atoms, one must test positivity of polynomials $\fK_{\cG}(x)$ where $x_v = 1 - 1/s^{(v)}$.    Hence, to understand how the atoms behave as the parameters $\alpha^{(v)}$ and $\lambda_j^{(v)}$ vary, one needs to understand the region in $[0,1]^{\cV}$ where the polynomials $\fK_{\cG'}(x)$ are positive for all induced subgraphs $\cG' \subseteq \cG$.  We denote this region by
				\begin{equation} \label{eq: region of G}
					R(\cG) := \{x \in [0,1]^{\cV}: \fK_{\cG'}(x|_{\cG'}) > 0 \text{ for induced } \cG' \subseteq \cG \}.
				\end{equation}
				We also remark that
				\begin{equation} \label{eq: region of G}
					R(\cG) := \{x \in [0,\infty)^{\cV}: \fK_{\cG'}(x|_{\cG'}) > 0 \text{ for induced } \cG' \subseteq \cG \};
				\end{equation}
				indeed, if $\fK_{\cG'}(x|_{\cG'}) > 0$ for induced $\cG' \subseteq \cG$, then by applying this to $\cG' = \{v\}$, we obtain $1 - x_v > 0$ and so $x_v < 1$; hence, $x \in [0,1]^{\cV}$ automatically.  In this section, we will show that $R(\cG)$ is the connected component of $[0,\infty)^{\cV} \cap \{\fK_{\cG} > 0\}$ containing $0$, and in many cases the boundary in the positive orthant is smooth.  We start by computing the derivatives of $\fK_{\cG}$ in hopes of applying the implicit function theorem.
				
				\begin{lemma} \label{lem: derivative computation}
					Let $\cG$ be a graph and let $x = (x_v)_{v \in \cV} \in \C^{\cV}$.  For each vertex $j$, let $\mathcal{S}(j)$ be the induced subgraph with vertex set $\{i: (i,j) \in \cE  \}$.  Then
					\[
					\partial_j \fK_{\cG}(x) = -\fK_{\mathcal{S}(j)}(x|_{\mathcal{S}(j)}), \text{ where } x|_{\mathcal{S}(j)} = (x_v)_{v \in \mathcal{S}(j)}.
					\]
				\end{lemma}
				
				\begin{proof}
					Recall $\fK_{\cG}(x)$ has one term $(-1)^{|\cK|} \prod_{v \in \cK} x_v$ for each clique $\cK$.  If $\cK$ does not contain vertex $j$, then the derivative with respect to $x_j$ is zero.  If $\cK$ does contain $j$, then by definition every vertex $v$ in $\cK$ other than $j$ must be in $\mathcal{S}(j)$, and so $\cK = \cK' \cup \{j\}$, where $\cK'$ is a clique in $\mathcal{S}(j)$.  Differentiating with respect to $x_j$ then produces $-(-1)^{|\cK'|} \prod_{v \in \cK'} x_v$, which yields exactly all the terms in $\fK_{\mathcal{S}(j)}(x|_{\mathcal{S}(j)})$.
				\end{proof}
				
				We can now determine the topological behavior of $R(\cG)$.
				
				\begin{lemma} \label{lem: homeomorphism}
					For a unit vector $u$ in the nonnegative orthant of $\R^{\cV}$, let
					\[
					\rho(u) = \inf \{r \in [0,\infty): \min_{\cG' \subseteq \cG} \fK_{\cG'}(ru|_{\mathcal{G}'}) = 0\},
					\]
					where the minimum is taken over all induced subgraphs $\cG'$.  Then
					\begin{enumerate}
						\item $\rho(u) = \min \{r \in [0,\infty): \fK_{\cG}(ru) = 0 \}$.
						\item If $x \in [0,1]^{\cV} \setminus \set 0$, then $x \in R(\cG)$ if and only if $|x| < \rho(x / |x|)$.
						\item The mapping $u \mapsto \rho(u)$ is continuous.
						\item There is a homeomorphism
						\[
						\{y \in [0,1]^{\cV}: |y| < 1 \} \to R(\cG): y \mapsto \rho(y/|y|) y,
						\]
						where the right-hand side is interpreted as $0$ when $y = 0$.
					\end{enumerate}
				\end{lemma}
				
				\begin{proof}
					(1) Let $\widehat{\rho}(u) = \inf \{r \in [0,\infty): \fK_{\cG}(ru) = 0 \}$.  By construction, $\rho(u) \leq \widehat{\rho}(u)$.  For the opposite inequality, note that $\fK_{\cG}(ru) > 0$ for $r < \widehat{\rho}(u)$, and therefore by \ref{lem: convergence criteria} (5) $\implies$ (3), we have $\fK_{\cG}(y) > 0$ for every $y \in \prod_{v \in \cV} \{0, ru_v\}$ and $r < \widehat{\rho}(u)$.  As we observed in the proof of Theorem \ref{thm: classification of finite-dimensional summands}, replacing some of the variables with zero is equivalent to evaluating $\fK_{\cG'}$ for an induced subgraph $\cG'$.  Hence, we have $\fK_{\cG'}(ru|_{\cG'}) > 0$ for $r < \widehat{\rho}(u)$ and for all induced subgraphs $\cG'$.  Therefore, we also have $\rho(u) > r$ whenever $\widehat{\rho}(u) > r$, and thus $\widehat{\rho}(u) \leq \rho(u)$.
					
					It remains to show that infimum is actually a minimum.  Looking at subgraphs with only one vertex, we see that if the $v$th coordinate of $ru$ is equal to $1$, then $\fK_{\set v}(ru_v) = 1 - ru_v = 0$.  Thus, $\rho(u)u$ must be in $[0,1]^{\cV}$ and in particular $\rho(u) \leq \sqrt{|\cV|}$.  Thus, $\{r \in [0,\infty): \fK_{\cG}(ru) = 0\}$ is a nonempty compact set and so the infimum is achieved.
					
					(2) This follows from part (1) of this lemma together with Lemma \ref{lem: convergence criteria} (3) $\iff$ (5).
					
					(3) First, we show $\rho$ is lower semi-continuous, or in other words for each $r^* > 0$ the set $\{u: \rho(u) > r^* \}$ is open.  Let $u^*$ be in this set.  Then the function $\gamma(r,u) = \min_{\cG'} \fK_{\cG'}(ru|_{\cG'})$ is strictly positive on $[0,r^*] \times \{u^*\}$.  By uniform continuity of $\gamma$ on compact sets, there is neighborhood of $O$ of $u^*$ such that $\gamma > 0$ on $[0,r^*] \times O$.  Then for $u \in O$, we have $\rho(u) > r^*$, as desired.
					
					Now we need to show that $\{u: \rho(u) < r^*\}$ is open for each $r^*$. Fix $u^*$ in this set.  Then $\fK_{\cG'}(\rho(u^*)u^*|_{\cG'}) = 0$ for some induced subgraph $\cG' \subseteq \cG$.  Choose such a $\cG'$ which is minimal, and consider differentiating $\fK_{\cG'}$ with respect to $r$.  By the chain rule,
					\[
					\frac{d}{dr} \Bigr|_{r = \rho(u^*)} \fK_{\cG'}(r u^*|_{\cG'}) = \sum_{j \in \cV'} u_j^* [\partial_j \fK_{\cG'}](\rho(u^*)u^*|_{\cG'}).
					\]
					For $j \in \cV'$, Lemma \ref{lem: derivative computation} shows that
					\[
					\partial_j \fK_{\cG'}(x|_{\cG'}) = -\fK_{\cG' \cap \cS(j)}(x|_{\cG' \cap \cS(j)}),
					\]
					which is nonzero at $\rho(u^*) u^*$ by minimality of $\cG'$ (here note that $\cG' \cap \cS(j)$ does not contain $j$).  It is in fact strictly negative; indeed, since $\rho(u^*) u^*$ is in the closure of $R(\cG)$, the term $\fK_{\cG' \cap \cS(j)}(\rho(u^*) u^*|_{\cG'\cap \cS(j)})$ cannot be negative.  We finally note that $u^*$ must have some nonzero coordinate in $\cV'$ since otherwise $\fK_{\cG'}(\rho(u^*)u^*|_{\cG'})$ would equal $1$.  Therefore, overall
					\[
					\frac{d}{dr} \Bigr|_{r = \rho(u^*)} \fK_{\cG'}(r u^*|_{\cG'}) = -\sum_{j \in \cV'} u_j^* \fK_{\cG' \cap \cS(j)}(\rho(u^*)u^*|_{\cG' \cap \cS(j)}) < 0,
					\]
					since we have a sum of nonpositive terms at least one of which is strictly negative.  Therefore, by the implicit function theorem, there exists a neighborhood $O$ of $u^*$ in the unit sphere and a smooth function $\widehat{\rho}: O \to (0,\infty)$ such that
					\[
					\fK_{\cG'}(\widehat{\rho}(u) u|_{\cG'}) = 0
					\]
					and $\widehat{\rho}(u^*) = \rho(u^*)$.  For $u \in [0,\infty)^{\cV} \cap O$, we have $\rho(u) \leq \widehat{\rho}(u)$ by definition of $\rho$, and hence
					\[
					\{u: \rho(u) < r^*\} \supseteq \{u \in O \cap [0,\infty)^{\cV}: \widehat{\rho}(u) < r^* \}.
					\]
					The right-hand side is open by continuity of $\widehat{\rho}$ from the implicit function theorem and contains $u^*$ since $\widehat{\rho}(u^*) = \rho(u^*) < r^*$.  Thus, $\{u: \rho(u) < r^*\}$ contains an open neighborhood of $u^*$.  Since $u^*$ was arbitrary $\{u: \rho(u) < r^*\}$ is open.
					
					(4) This follows readily from (2) and (3) by basic topology.
				\end{proof}

				\begin{proposition} \label{prop: component}
					The region $R(\cG)$ is the connected component of $\{x \in [0,1]^{\cV}: \fK_{\cG}(x) > 0 \}$ containing $0$.
				\end{proposition}
				
				\begin{proof}
					Let $O = \{x \in [0,1]^{\cV}: \fK_{\cG}(x) > 0\}$.  Then $R(\cG) \subseteq O$ by definition.  Lemma \ref{lem: homeomorphism} (4) implies that $R(\cG)$ is connected, and so $R(\cG)$ is contained in the connected component of $O$ that contains $0$.
					
					Conversely, consider the open sets
					\[
					U = \{x \in [0,1]^{\cV}: x \neq 0, |x| < \rho(x/|x|) \}, \qquad \cV = \{x \in [0,1]^{\cV}: x \neq 0, |x| > \rho(x/|x|) \}.
					\]
					By Lemma \ref{lem: homeomorphism} (3), $U$ and $V$ are disjoint open sets.  By Lemma \ref{lem: homeomorphism} (1), if $|x| = \rho(x/|x|)$, then $\fK_{\cG}(x) = 0$.  Therefore, $O \setminus \set 0 \subseteq U \cup V$.  Also, $U$ contains a punctured neighborhood of the origin since $\fK_{\cG}(0) = 1 > 0$.  Hence, the connected component of $0$ in $O$ is contained in $U \cup \set 0$.  By Lemma \ref{lem: homeomorphism} (2), $U = R(\cG) \setminus \set 0$.
				\end{proof}
				
				Next, we consider whether the boundary of $R(\cG)$ is smooth.  For this, we need to analyze more carefully when the derivatives of $\fK_{\cG}$ can vanish.

				\begin{lemma} \label{lem:joincomponent}
					Let $\cG$ be a graph and $x \in (0,\infty)^{\cV}$ with
					\[
					\min_{\cG' \subseteq \cG} \fK_{\cG'}(x|_{\cG'}) = 0,
					\]
					where $\cG'$ runs over all induced subgraphs of $\cG$.  Let $\cG_0 = (\cV_0,\cE_0)$ be a minimal induced subgraph of $\cG$ such that $\fK_{\cG_0}(x|_{\cG_0}) = 0$.  Then $\cG$ decomposes as a graph join of $\cG_0$ with the complementary subgraph $\cG_1 = (\cV_1,\cE_1)$.
				\end{lemma}
				
				\begin{proof}
					It suffices to show that every vertex $w \in \cG_1$ is adjacent to every vertex in $\cG_0$.  Let $w \in \cV_1$, and write $\cG_w$ for the subgraph induced by $\cV_0 \cup \set w$.  Observe that
					\[
					\fK_{\cG_w}(x|_{\cG_w}) = \fK_{\cG_0}(x|_{\cG_0}) - x_w \fK_{\cS(w) \cap \cG_0}(x|_{\cS(w) \cap \cG_0});
					\]
					indeed, every clique $\cK$ in $\cG_w$ either (a) does not contain $w$ and hence appears in $\fK_{\cG_0}(x|_{\cG_0})$ or (b) contains $w$ and so is $\cK' \cup \set w$ for some clique $\cK'$ in $\cS(w) \cap \cG_0$ and hence appears in $- x_w \fK_{\cS(w) \cap \cG_0}(x|_{\cS(w) \cap \cG_0})$ (as in Lemma \ref{lem: derivative computation}).  By assumption,
					\[
					\fK_{\cG_w}(x|_{\cG_w}) \geq 0, \qquad \fK_{\cG_0}(x|_{\cG_0}) = 0, \qquad -x_w \fK_{\cS(w) \cap \cG_0}(x|_{\cS(w) \cap \cG_0}) \leq 0.
					\]
					Therefore, the above equality forces that
					\[
					\fK_{\cG_w}(x|_{\cG_w}) = 0 = x_w \fK_{\cS(w) \cap \cG_0}(x|_{\cS(w) \cap \cG_0}).
					\]
					Since $x_w > 0$, we have $\fK_{\cS(w) \cap \cG_0}(x|_{\cS(w) \cap \cG_0}) = 0$.  By minimality of $\cG_0$, we have $\cS(w) \cap \cG_0 = \cG_0$, or $\cG_0 \subseteq \cS(w)$ which means that $w$ is adjacent to every vertex in $\cG_0$ as desired.
				\end{proof}
				
				\begin{proposition} \label{prop: irreducible smoothness}
					Let $\cG$ be a graph which is join-irreducible.  Then the derivatives of $\fK_{\cG}$ are nonvanishing on $\partial R(\cG) \cap (0,\infty)^{\cV}$, and hence this is a smooth submanifold of $(0,\infty)^{\cV}$.
				\end{proposition}
				
				\begin{proof}
					Let $x \in \partial R(\cG) \cap (0,\infty)^{\cV}$.  In light of Lemma \ref{lem: homeomorphism}, we must have $\fK_{\cG}(x) = 0$.  If $\cG_0$ is a minimal subgraph where $\fK_{\cG_0}$ vanishes, then $\cG$ decomposes as a join of $\cG_0$ and its complement, which forces $\cG_0 = \cG$ since $\cG$ is irreducible.  Therefore, $\fK_{\cG'}$ is non-vanishing for any strictly smaller subgraph.  In particular, by Lemma \ref{lem: derivative computation}, the derivatives of $\fK_{\cG}$ are non-vanishing.  Hence, the implicit function theorem implies smoothness of $\partial R(\cG) \cap (0,\infty)^{\cV}$.
				\end{proof}
				
				In case $\cG$ is not join-irreducible, we can analyze the graph using the join decomposition.
				
				\begin{lemma} \label{lem: join decomposition}
					Let $\cG$ be a graph with a join decomposition $\cG = \cG_1 + \dots + \cG_k$.  Then
					\[
					\fK_{\cG}(x) = \prod_{j=1}^k \fK_{\cG_j}(x|_{\cG_j}),
					\]
					and
					\[
					R(\cG) = R(\cG_1) \times \dots \times R(\cG_k).
					\]
				\end{lemma}
				
				\begin{proof}
					If $\cK$ is a clique in $\cG$, then $\cK \cap \cV_j$ is a clique in $\cG_j$ for each $j$.  Conversely, if $\cK_j$ is a clique in $\cV_j$, then by construction of the graph join, $\cK = \cK_1 \cup \dots \cup \cK_k$ is a clique in $\cG$.  Moreover,
					\[
					(-1)^{|\cK|} \prod_{v \in \cK} x_v = \prod_{j=1}^k \left[(-1)^{|\cK_j|} \prod_{v \in \cK_j} x_v \right].
					\]
					Hence, the sum over all cliques in $\cG$ decomposes as the product of sums associated to each $\cG_j$ which yields the asserted identity for the polynomials.  The identity for the regions then follows from Lemma \ref{lem: homeomorphism} (1) and (2).
				\end{proof}
				
				From this we can see that $R(\cG) \cap (0,\infty)^{\cV}$ is the product of open sets with smooth boundary, and the only way that $\partial R(\cG)$ can be non-smooth at a point $x \in (0,\infty)^{\cV}$ is if there are two separate indices $i$ and $j$ such that $x|_{\cG_i}$ and $x|_{\cG_j}$ are both boundary points of $R(\cG_i)$ and $R(\cG_j)$ respectively.  Meanwhile, if some of the coordinates of $x$ are allowed to be zero, then more complicated behavior can occur.  Thus, in order to check whether the derivatives $\partial_j \fK_{\cG}(x)$ vanishes, we have to examine each vertex of $\cG$ in relation with the subgraph where the values of $x_v$ are positive.
				
				\begin{lemma} \label{lem:irrelevantzero}
					Let $\cG$ be a graph and $x \in [0,\infty)^{\cV}$. 
					Let $\cG^+$ be the subgraph induced by vertices $v$ where $x_v > 0$.  Moreover, for $j \in \cV$, let $\cG_j^+$ be the subgraph induced by $j$ together with $\cG^+$.  For each $j \in \cV$, we have
					\[
					\fK_{\cG_+}(x|_{\cG_+}) = \fK_{\cG_j^+}(x|_{\cG_j^+}) =  \fK_{\cG}(x)
					\]
					and
					\[
					\partial_j \fK_{\cG}(x) = \partial_j \fK_{\cG_j^+}(x|_{\cG_j^+}).
					\]
				\end{lemma}
				
				\begin{proof}
					For the first claim, since $x_v = 0$ for $v$ not in $\cG_j^+$, all the terms in $\fK_{\cG}(x)$ from cliques that are not contained in $\cG_j^+$ vanish, leaving exactly the terms in $\fK_{\cG^+}(x|_{\cG_+})$.  Thus, $\fK_{\cG}(x) = \fK_{\cG^+}(x|_{\cG_+})$.  The same reasoning applies with $\cG_j^+$ in place of $\cG^+$.
					
					For the second claim, use Lemma \ref{lem: derivative computation} together with the first claim:
					\begin{equation} \label{eq: derivative reduction}
						\partial_j \fK_{\cG}(x) = -\fK_{\cS(j)}(x|_{\cS(j)}) = -\fK_{\cS(j) \cap \cG_j^+}(x|_{\cS(j) \cap \cG_j^+}) = \partial_j \fK_{\cG_j^+}(x|_{\cG_j^+}).  \qedhere
					\end{equation}
				\end{proof}
				
				\begin{proposition} \label{prop: final derivative vanishing}
					Let $\cG$ be a graph and $x \in [0, \infty)^{\cV} \cap \partial R(\cG)$.  (Here the boundary is considered relative to $[0,\infty)^{\cV}$ rather than $\R^{\cV}$).  Let $\cG^+$ and $\cG_j^+$ be as in Lemma \ref{lem:irrelevantzero}.  For each $j \in \cV$, the following are equivalent:
					\begin{enumerate}[(1)]
						\item $\partial_j \fK_{\cG}(x) = 0$.
						\item $\cG_j^+$ decomposes as a non-trivial graph join $\cG_j^+ = \mathcal{A} + \mathcal{B}$ with $j \in \mathcal{A}$ and $\fK_{\mathcal{B}}(x|_{\mathcal{B}}) = 0$.
					\end{enumerate}
					In particular, if $\cG_j^+$ is irreducible, then $\partial_j \fK_{\cG}(x)$ is nonzero.
				\end{proposition}
				
				\begin{proof}
					(1) $\implies$ (2).  By \eqref{eq: derivative reduction}, we have $0 = \partial_j \fK_{\cG}(x) = -\fK_{\cS(j) \cap \cG_j^+}(x|_{\cS(j) \cap \cG_j^+})$.  Note that $\cS(j) \cap \cG_j^+ = \cS(j) \cap \cG^+ \subseteq \cG^+$.  Therefore, there exists a minimal induced subgraph $\mathcal{B} \subseteq \cS(j) \cap \cG^+$ such that $\fK_{\cB}(x|_{\cB}) = 0$.  By applying Lemma \ref{lem:joincomponent} to $\cG^+$, we see that $\mathcal{G}^+$ decomposes as the join of $\mathcal{B}$ with another subgraph $\mathcal{A}_0$.  By construction, $j$ is adjacent to all the vertices in $\mathcal{B}$; thus, letting $\mathcal{A}$ the induced subgraph obtained by adding $j$ into $\mathcal{A}_0$, we have that $\cG_j^+ = \mathcal{A} + \mathcal{B}$.
					Since $\fK_{\mathcal{B}}(x|_{\cB}) = 0$ and $\fK_{\varnothing}(\cdot) = 1$, we must have $\cB \neq \varnothing$ so the join decomposition is non-trivial.
					
					(2) $\implies$ (1).  By Lemma \ref{lem:irrelevantzero}, we have $\partial_j \cK_{\mathcal{G}}(x) = \fK_{\mathcal{G}_j^+}(x|_{\mathcal{G}_j^+})$ at the given point $x$.  By Lemma \ref{lem: join decomposition}, we have
					\[
					\fK_{\mathcal{G}_j^+}(y|_{\mathcal{G}_j^+}) = \fK_{\mathcal{A}}(y|_{\mathcal{A}}) \fK_{\mathcal{B}}(y|_{\mathcal{B}}) \text{ for } y \in [0,\infty)^{\cV}.
					\]
					By assumption, $\fK_{\mathcal{B}}(x|_{\mathcal{B}}) = 0$, and, since $j \in \mathcal{A}$, the value of $\fK_{\mathcal{B}}(x|_{\mathcal{B}})$ remains unchanged if we perturb the $j$th coordinate of $x$.  Therefore, $\fK_{\mathcal{G}_j^+}(y|_{\mathcal{G}_j^+}) = 0$ whenever $y$ is a perturbation of $x$ in only the $j$th coordinate.  Hence, $\partial_j \fK_{\mathcal{G}_j^+}(x|_{\mathcal{G}_j^+}) = 0$ as desired.
					
					The ``in particular'' assertion is immediate from (1) $\implies$ (2).
				\end{proof}
				
				\begin{remark}
					Using Proposition \ref{prop: final derivative vanishing}, we can test when all the derivatives of $\fK_{\mathcal{G}}(x)$ vanish.  This happens if and only if, for every $j$, $\mathcal{G}_j^+$ has a decomposition as in Proposition \ref{prop: final derivative vanishing} (2).  In particular, this means that $\cG_j^+$ has at least two vertices, and so $\cG^+$ is nonempty.  Then by applying this condition to vertices $j \in \cG^+$, we conclude that $\cG^+$ must be reducible, and so $\cG^+ = \mathcal{C}_1 + \dots + \mathcal{C}_m$ for some $m \geq 2$.  Now, for each $j$, we have a decomposition $\mathcal{G}_j^+ = \mathcal{A}_j + \mathcal{B}_j$ given by condition (2).  Each of the irreducible components $\mathcal{C}_i$ from $\mathcal{G}^+$ must be contained in either $\mathcal{A}_j$ or $\mathcal{B}_j$ (with at least one being contained in $\mathcal{B}_j$) for each $j$.  This yields for each $j$, a partition of $\{\mathcal{C}_1,\dots,\mathcal{C}_m\}$ into two parts.  However, it could happen that different values of $j$ in $\mathcal{G} \setminus \mathcal{G}^+$ produce completely different partitions of $\{\mathcal{C}_1,\dots,\mathcal{C}_m\}$.  We also have no control over whether two vertices in $\mathcal{G} \setminus \mathcal{G}^+$ are adjacent.  Hence, $\mathcal{G}_j^+$ can decompose as a graph join as in the proposition, for every $j$, without $\cG$ itself necessarily decomposing as a graph join.
				\end{remark}
				
				\bibliographystyle{amsalpha}
				\bibliography{bib-graph-atoms}

\newcommand{\etalchar}[1]{$^{#1}$}
\providecommand{\bysame}{\leavevmode\hbox to3em{\hrulefill}\thinspace}
\providecommand{\MR}{\relax\ifhmode\unskip\space\fi MR }
\providecommand{\MRhref}[2]{%
  \href{http://www.ams.org/mathscinet-getitem?mr=#1}{#2}
}
\providecommand{\href}[2]{#2}
\begin{thebibliography}{CdSH{\etalchar{+}}25}

\bibitem[BBL21]{BelBerLiu2021}
Serban~T. Belinschi, Hari Bercovici, and Weihua Liu, \emph{The atoms of
  operator-valued free convolutions}, Journal of Operator Theory \textbf{85}
  (2021), no.~1, 303--320.

\bibitem[Bla06]{Blackadar2006encyclopedia}
Bruce Blackadar, \emph{Operator algebras: Theory of $c^*$-algebras and von
  {N}eumann algebras}, Encyclopaedia of Mathematical Sciences, vol. 122,
  Springer-Verlag, Berlin, 2006. \MR{2188261}

\bibitem[Boc93]{Boca1993}
Florin Boca, \emph{Completely positive maps on amalgamated product
  ${C}^*$-algebras.}, MATHEMATICA SCANDINAVICA \textbf{72} (1993), 212--222.

\bibitem[BV98]{BV1998}
Hari Bercovici and Dan Voiculescu, \emph{Regularity questions for free
  convolution}, pp.~37--47, Birkh{\"a}user Basel, Basel, 1998.

\bibitem[Cas20]{Caspers2020}
Martijn Caspers, \emph{Absence of {Cartan} subalgebras for right-angled {Hecke}
  von {Neumann} algebras}, no.~1, 1--28.

\bibitem[CdSH{\etalchar{+}}25]{tesseract}
Ian Charlesworth, Rolando de~Santiago, Ben Hayes, David Jekel,
  Srivatsav~Kunnawalkam Elayavalli, and Brent Nelson, \emph{On the structure of
  graph product von {Neumann} algebras}, Publ. Res. Inst. Math. Sci.
  \textbf{61} (2025), no.~4, 713--762.

\bibitem[CF69]{CF1969}
P.~Cartier and D.~Foata, \emph{Probl\`emes combinatoires de commutation et
  r\'earrangements}, Lecture Notes in Mathematics, vol. No. 85,
  Springer-Verlag, Berlin-New York, 1969. \MR{239978}

\bibitem[CF17]{CaFi2017}
Martijn Caspers and Pierre Fima, \emph{Graph products of operator algebras}, J.
  Noncommut. Geom \textbf{11} (2017), no.~1, 367--411.

\bibitem[CKL21]{CKL2021}
M.~Caspers, M.~Klisse, and N.~S. Larsen, \emph{Graph product {Khintchine}
  inequalities and {Hecke} $\mathrm{C}^*$-algebras: {Haagerup} inequalities,
  (non)simplicity, nuclearity and exactness}, J. Funct. Anal. \textbf{280}
  (2021), no.~1, 108795.

\bibitem[CSY]{COSY2024}
Guillaume C{\'e}bron, Patrick~Oliveira Santos, and Pierre Youssef,
  \emph{Graphon-theoretic approach to central limit theorems for
  $\varepsilon$-independence}, preprint, arXiv:2411.13062.

\bibitem[Dyk93]{Dykema1993freedimension}
Ken Dykema, \emph{{Free products of hyperfinite von Neumann algebras and free
  dimension}}, Duke Mathematical Journal \textbf{69} (1993), no.~1, 97--119.

\bibitem[Gar16]{Garncarek2016}
{\L}.~Garncarek, \emph{Factoriality of {Hecke}--von {Neumann} algebras of
  right-angled {C}oxeter groups}, J. Funct. Anal. \textbf{270} (2016), no.~3,
  1202--1219.

\bibitem[GJ04]{GJ2004}
Ian~P. Goulden and David~M. Jackson, \emph{Combinatorial enumeration}, Dover
  Publications, Inc., Mineola, NY, 2004, With a foreword by Gian-Carlo Rota,
  Reprint of the 1983 original. \MR{2079788}

\bibitem[Gre90]{Green1990}
Elisabeth~Ruth Green, \emph{Graph products of groups}, Ph.D. thesis, University
  of Leeds, 1990.

\bibitem[HM95]{HM1995}
S.~Hermiller and J.~Meier, \emph{Algorithms and geometry for graph products of
  groups}, Journal of Algebra \textbf{171} (1995), no.~1, 230--257.

\bibitem[JS97]{JonesSunder1997}
V.~Jones and V.~S. Sunder, \emph{Introduction to subfactors}, London
  Mathematical Society Lecture Note Series, Cambridge University Press, 1997.

\bibitem[JV13]{MR3071084}
Matthieu Josuat-Verg\`es, \emph{Cumulants of the {$q$}-semicircular law,
  {T}utte polynomials, and heaps}, Canad. J. Math. \textbf{65} (2013), no.~4,
  863--878. \MR{3071084}

\bibitem[Maz95]{MR1478993}
Antoni Mazurkiewicz, \emph{Introduction to trace theory}, The book of traces,
  World Sci. Publ., River Edge, NJ, 1995, pp.~3--41. \MR{1478993}

\bibitem[ML19]{MorLau2019}
Siddhardh~C. Morampudi and Chris~R. Laumann, \emph{Many-body systems with
  random spatially local interactions}, Phys. Rev. B \textbf{100} (2019),
  245152.

\bibitem[M{\l}o04]{Mlot2004}
W.~M{\l}otkowski, \emph{$\lambda$-free probability}, Infinite Dimensional
  Analysis. Quantum Probability and Related Topics \textbf{7} (2004), no.~1,
  27--41.

\bibitem[MSW17]{MaiSpeicherWeber2017}
Tobias Mai, Roland Speicher, and Moritz Weber, \emph{Absence of algebraic
  relations and of zero divisors under the assumption of full non-microstates
  free entropy dimension}, Advances in Mathematics \textbf{304} (2017),
  1080--1107.

\bibitem[OTY25]{OSTY2025}
Patrick {Oliveira Santos}, Raghavendra Tripathi, and Pierre Youssef,
  \emph{Khintchine inequalities, trace monoids and {Tur\'an}-type problems},
  Preprint, arxiv.org/abs/2506.02517, 2025.

\bibitem[RS23]{RaumSkalski2023}
Sven Raum and Adam Skalski, \emph{Factorial multiparameter {Hecke} von
  {Neumann} algebras and representations of groups acting on right-angled
  buildings}, Journal de Math{\'e}matiques Pures et Appliqu{\'e}es \textbf{172}
  (2023), 265--298.

\bibitem[Sak71]{Sakai1971}
Sh{\^o}ichir{\^o} Sakai, \emph{$\mathrm{C}^*$-algebras and
  $\mathrm{W}^*$-algebras}, Ergebnisse der {M}athematik und ihrer
  {G}renzgebiete, vol.~60, Springer-Verlag, Berlin Heidelberg, 1971.

\bibitem[SW16]{SpWy2016}
Roland Speicher and Janusz Wysocza{\'n}ski, \emph{Mixtures of classical and
  free independence}, Archiv der Mathematik \textbf{107} (2016), 445--453.

\bibitem[Ued11]{UedaTypeIIIfreeproduct}
Yoshimichi Ueda, \emph{Factoriality, type classification and fullness for free
  product von {N}eumann algebras}, Adv. Math. \textbf{228} (2011), no.~5,
  2647--2671. \MR{2838053}

\bibitem[Vie89]{MR1110852}
G\'erard~Xavier Viennot, \emph{Heaps of pieces. {I}. {B}asic definitions and
  combinatorial lemmas}, Graph theory and its applications: {E}ast and {W}est
  ({J}inan, 1986), Ann. New York Acad. Sci., vol. 576, New York Acad. Sci., New
  York, 1989, pp.~542--570. \MR{1110852}

\end{thebibliography}
				
			\end{document}